\documentclass[aos,preprint]{imsart}

\usepackage{amsmath,amssymb,fullpage,fouriernc,color,url,graphicx}
\usepackage[authoryear]{natbib}
\usepackage{amsthm}
\usepackage[colorlinks,citecolor=blue,urlcolor=blue]{hyperref}
\usepackage[all]{xy}

\usepackage{bbm}

\let\hat\widehat
\let\tilde\widetilde

\newcommand{\dtm}{\delta}
\newcommand{\edtm}{\widehat\delta}
\newcommand{\dimension}{d}

\newcommand{\figref}[1]{Figure~\ref{#1}}
\newcommand{\secref}[1]{Section~\ref{#1}}
\newcommand{\thmref}[1]{Theorem~\ref{#1}}

\renewcommand{\P}{\mbox{$\mathbb{P}$}}
\newcommand{\R}{\mbox{$\mathbb{R}$}}
\newcommand{\E}{\mbox{$\mathbb{E}$}}

\DeclareMathOperator*{\argmax}{argmax}

\newtheorem{theorem}{Theorem}
\newtheorem{lemma}[theorem]{Lemma}
\newtheorem{example}[theorem]{Example}
\newtheorem{corollary}[theorem]{Corollary}

\newtheorem{remark}[theorem]{Remark}

\begin{document}

\begin{frontmatter}

\title{Robust Topological Inference:\\Distance To a Measure and Kernel Distance}
\runtitle{Robust Topological Inference}

\begin{aug}
\author{\fnms{Fr\'ed\'eric}
\snm{Chazal}\thanksref{a,t1}\ead[label=e1]{frederic.chazal@inria.fr}},
\author{\fnms{Brittany Terese}
\snm{Fasy}\thanksref{b}\ead[label=e2]{brittany@fasy.us}},
\author{\fnms{Fabrizio}
\snm{Lecci}\thanksref{c,t3}\ead[label=e3]{lecci@cmu.edu}},\\
\author{\fnms{Bertrand}
\snm{Michel}\thanksref{d,t1}\ead[label=e4]{bertrand.michel@upmc.fr}},
\author{\fnms{Alessandro}
\snm{Rinaldo}\thanksref{c,t3}\ead[label=e5]{arinaldo@cmu.edu}},
\and
\author{\fnms{Larry} 
\snm{Wasserman}\thanksref{c,t5}\ead[label=e6]{larry@cmu.edu}},

\thankstext{t1}{Research supported by ANR-13-BS01-0008.}
\thankstext{t3}{Research supported by NSF CAREER Grant DMS 1149677.}
\thankstext{t5}{Research supported by Air Force Grant FA95500910373, NSF Grant
DMS-0806009.}

\runauthor{Chazal et al.}

\affiliation[a]{INRIA}
\affiliation[b]{Tulane University}
\affiliation[c]{Carnegie Mellon University}
\affiliation[d]{Universit\'e Pierre et Marie Curie}

\address{
F. Chazal\\
Inria Saclay - Ile-de-France\\
Alan Turing Bldg, Office 2043\\
1 rue Honor\'e d'Estienne d'Orves\\
 91120 Palaiseau \\
\printead{e1}
}

\address{
B.T. Fasy\\
Department of Computer Science\\
Tulane University\\
New Orleans, LA 70118\\
\printead{e2}
}

\address{
B. Michel\\
LSTA\\
Universit\'e Pierre et Marie Curie\\
15-25, Office 220 \\
4 place Jussieu, 75005 Paris\\
\printead{e4}
}

\address{
F. Lecci\\
A. Rinaldo\\
L. Wasserman\\
Department of Statistics\\
Carnegie Mellon University\\
Pittsburgh, PA 15213\\
\printead{e3}\\
\printead{e5}\\
\printead{e6}
}

\end{aug}

\today

\begin{abstract}
Let $P$ be a distribution with support
$S$. 
The salient features of $S$
can be quantified with persistent homology,
which summarizes topological features of the sublevel sets of
the distance function
(the distance of any point $x$ to $S$).
Given a sample from $P$
we can infer 
the persistent homology
using an empirical version
of the distance function.
However, the empirical distance function
is highly non-robust to noise and outliers.
Even one outlier is deadly.
The distance-to-a-measure (DTM),
introduced by \cite{chazal2011geometric},
and the kernel distance, introduced by \cite{phillips2014goemetric},
are smooth functions that 
provide useful topological information 
but are robust
to noise and outliers.
\cite{massart2014} derived concentration bounds
for DTM.
Building on these results, we derive
limiting distributions and confidence sets, and 
we propose a method for choosing tuning parameters.
\end{abstract}


\begin{keyword}
\kwd{persistent homology}
\kwd{topology}
\kwd{density estimation.}
\end{keyword}

\end{frontmatter}

\section{Introduction}

Figure~\ref{fig:introExample}
shows three complex point clouds,
based on a model used for simulating
cosmology data.
Visually, the three samples look very similar.
Below the data plots are the
persistence diagrams, which are summaries
of topological features defined in Section 2.
The persistence diagrams make it clearer
that the third data set is from a different data generating process
than the first two.

This is an example of how
topological features can summarize structure in point clouds.
The field of topological data analysis (TDA)
is concerned with defining such
topological features;
see \cite{carlsson2009topology}.
When performing TDA,
it is important to use topological measures that are
robust to noise.
This paper explores some of these robust topological measures.

Let $P$ be a distribution with compact support $S \subset \R^{\dimension}$.
One way to describe the shape of $S$ is by using homology.
Roughly speaking, the homology of $S$ measures the topological features of $S$,
such as the connected components, the holes, and the voids.
A more nuanced way to describe the shape of $S$ is using persistent homology,
which is a multiscale version of homology.
To describe persistent homology,
we begin with the distance function
$\Delta_S \colon \R^d \to \R$ for $S$ which is defined by
\begin{equation}
\Delta_S(x) = \inf_{y\in S} \Vert x-y \Vert.
\end{equation}
The sublevel sets
$L_t=\{ x:\ \Delta_S(x) \leq t\}$
provide multiscale topological information about $S$.
As~$t$ varies from zero to $\infty$,
topological features --- connected components, loops, voids ---
are born and die.
Persistent homology
quantifies the evolution of these topological features 
as a function of $t$.
See Figure~\ref{fig:circle}.
Each point on the persistence diagram represents the birth and death time
of a topological feature.

Given a sample
$X_1,\ldots, X_n\sim P$, the empirical distance function
is defined by
\begin{equation}
\hat \Delta(x) = \min_{X_i}\Vert x-X_i \Vert.
\end{equation}
If $P$ is supported on $S$,
and has a density bounded away from zero and infinity,
then $\hat \Delta$ is a consistent estimator of $\Delta_S$, i.e.,
$\sup_x |\hat \Delta(x) -\Delta_S(x)|\stackrel{P}{\to} 0.$
However, if there are outliers,
or noise,
then
$\hat \Delta(x)$ is no longer consistent.
Figure~\ref{fig:cassini} (bottom) shows that a few outliers
completely change the distance function.
In the language of robust statistics,
the empirical distance function has breakdown point~zero.

A more robust approach is to
estimate the persistent homology
of the super-level sets of the density $p$ of $P$.
As long as $P$ is concentrated near $S$,
we expect the level sets of $p$
to provide useful topological information about $S$.
Specifically, some level sets of $p$ are homotopic to~$S$ under weak 
conditions, and this implies that we can estimate the homology of $S$.
Note that, in this case, we are using the persistent homology of the 
super-level sets of $p$,
to estimate the homology of $S$.
This is the approach suggested by
\cite{bubenik2012statistical},
\cite{fasy2014statistical} and 
\cite{bobrowski2014topological}.
A related idea is to use persistent homology based on a kernel distance
\citep{phillips2014goemetric}.
In fact, the sublevel sets of the kernel distance are a rescaling
of the super-level sets of $p$,
so these two ideas are essentially equivalent.
We discuss this approach in Section \ref{section::kernels}.

A different approach, more closely related to the distance function,
but robust to noise, is to use the {\em distance-to-a-measure (DTM)},
$\dtm \equiv \dtm_{P,m}$, from \cite{chazal2011geometric}; see
Section~\ref{section::background}.  An estimate $\edtm$ of $\dtm$ is
obtained by replacing the true probability measure with the empirical
probability measure $P_n$, or with a deconvolved version of the
observed measure \cite{caillerie2011deconvolution}.  One then
constructs a persistence diagram based on the sub-level sets of the
DTM.  See Figure~\ref{fig:introExample}.  This approach is aimed at
estimating the persistent homology of $S$.  (The DTM also suggests new
approaches to density estimation; see \cite{biau2011weighted}.)

The density estimation approach and the DTM are both trying to
probe the topology of $S$.
But the former is using persistent homology to estimate the homology of $S$,
while
the DTM is directly trying to estimate the persistent homology of $S$.
We discuss this point in detail in 
Section \ref{sec::compare}.

\begin{figure}[!ht]
\begin{center}
\includegraphics[scale=0.31]{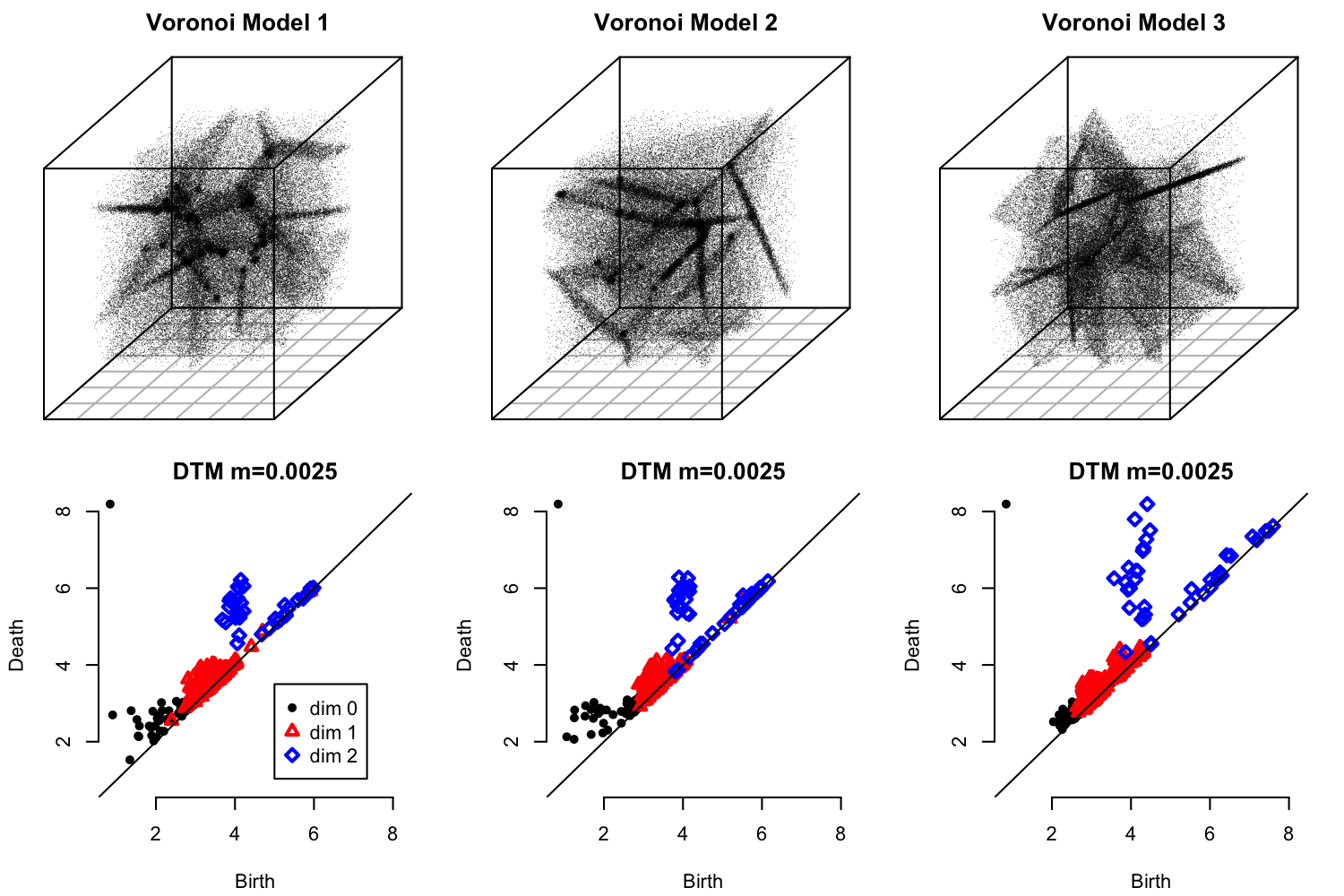}
\end{center}
\caption{The first two datasets come from the same data generating
  mechanism. In the third one, the particles are more concentrated
  around the walls of the Voronoi cells. Although the difference is
  not clear from the scatterplots, it is evident from the persistence
  diagrams of the sublevel sets of the distance-to-measure functions. 
  See Example \ref{ex:voronoi} for more details on the Voronoi Models.}
\label{fig:introExample}
\end{figure}

In this paper, we explore some statistical properties of these methods.
In particular:
\begin{enumerate}

\item We show that
$\sqrt{n}(\edtm^2(x)-\dtm^2(x))$ converges to a Gaussian process.
(\thmref{th:pointwiseDTM}).

\item We show that the bootstrap provides
asymptotically valid confidence bands for $\dtm$.
This allows us to 
identify significant
topological features.  
(\thmref{thm::limit}).

\item We find the limiting distribution of a key topological quantity called the
bottleneck distance. 
(\secref{ss:significance}).

\item We also show that, under additional assumptions, there is another version of the
bootstrap --- which we call the bottleneck bootstrap --- that provides more
precise inferences.  
(\secref{section::BB}).

\item We show similar results for the kernel distance.
(\secref{section::kernels}).

\item We propose a method for choosing the tuning parameter $m$
for DTM and the bandwidth~$h$ for the kernel distance.
(\secref{section::choosing}).

\item We show that the DTM and the KDE both suffer from boundary bias
and we suggest a method for reducing the bias.
(\secref{ss::bias}).

\end{enumerate}

{\em Notation.}
$B(x, \epsilon)$ is a Euclidean ball of radius $\epsilon$, centered at $x$. 
We define
$A\oplus\epsilon = \bigcup_{x\in A}B(x,\epsilon)$, the union of
$\epsilon$-balls centered at points in $A$.
If $x$ is a vector then
$||x||_\infty = \max_j |x_j|$.
Similarly, if $f$ is a real-valued fiction then
$||f||_\infty = \sup_x |f(x)|$.
We write
$X_n \rightsquigarrow X$ to mean that
$X_n$ converges in distribution to $X$, and we use symbols like
$c,C,\ldots,$ as generic positive constants.

{\bf Remark:}
The computing for the examples in this paper were done using the \textsf{R}
package \textbf{TDA}. See \cite{fasy2014introduction}. 
The package can be downloaded from
\url{http://cran.r-project.org/web/packages/TDA/index.html}.

{\bf Remark:}
In this paper, we discuss the DTM which uses a smoothing parameter $m$
and the kernel density estimator which uses a smoothing bandwidth $h$.
Unlike in traditional function estimation,
we do not send these parameters to zero as $n$ increases.
In TDA, the topological features created with a fixed smoothing parameter are of
interest.
Thus, all the theory in this paper treats the smoothing parameters as being bounded away from 0.
See also Section 4.4 in \cite{fasy2014statistical}.
In Section \ref{section::choosing}, we discuss the choice of these smoothing parameters.

\section{Background}
\label{section::background}

In this section, we define several distance functions and distance-like 
functions, and we introduce the relevant concepts from computational 
topology.  For more detail,
we refer the reader to 
\cite{edelsbrunner2010computational}.

\subsection{Distance Functions and Persistent Homology}

Let $S \subset\mathbb{R}^d$ be a compact set.
The {\em homology} of $S$ characterizes certain topological features of~$S$,
such
as its connected components, holes, and voids.
\emph{Persistent homology} is a multiscale version of homology.
Recall that the distance function
$\Delta_S$ for $S$ is
\begin{equation}
\Delta_S(x) = \inf_{y\in S} \Vert x-y \Vert.
\end{equation}
Let
$L_t=\{ x:\ \Delta_S(x) \leq t\}$.
We will refer to the parameter $t$ as ``time.''

Given the nested family of the sublevel sets of $\Delta_S$, the topology of $L_t$ changes as $t$ increases: new connected components can appear, existing connected components can merge, cycles and cavities can appear or be filled, etc.
Persistent homology tracks these changes, identifies \emph{features} and associates an \emph{interval} or \emph{lifetime} (from $t_\textrm{birth}$ to $t_\textrm{death}$) to them. For instance, a connected component is a feature that is born at the smallest $t$ such that the component is present in $L_t$, and dies when it merges with an older connected component. 
Intuitively, the longer a feature persists, the more relevant it is.

A feature, or more precisely its lifetime, can be represented as a segment whose extremities
have abscissae $t_\textrm{birth}$ and $t_\textrm{death}$; the
set of these segments is called the {\em barcode} of $\Delta_S$. An interval
can also be represented as a point in the plane with coordinates $(u,v) = (t_\textrm{birth},t_\textrm{death})$.
The set of points (with multiplicity) representing the
intervals is called the {\em persistence diagram} of $\Delta_S$. Note that
the diagram is entirely contained in the half-plane above the diagonal
defined by $u=v$, since death always occurs after birth.
This diagram is well-defined for any compact set $S$ (\cite{chazal2012structure}, Theorem 2.22). 
The most persistent features (supposedly the most important) are those represented by the points furthest from the diagonal in the diagram, whereas points close to the diagonal can be interpreted as (topological) noise.


Figure~\ref{fig:circle} shows a simple example.
Here, the points on the circle are regarded as a subset of $\mathbb{R}^2$.
At time zero, there is one connected component and one loop.
As $t$ increases, the loop dies.

\begin{figure}[t!b!]
\begin{center}
\includegraphics[scale=0.89]{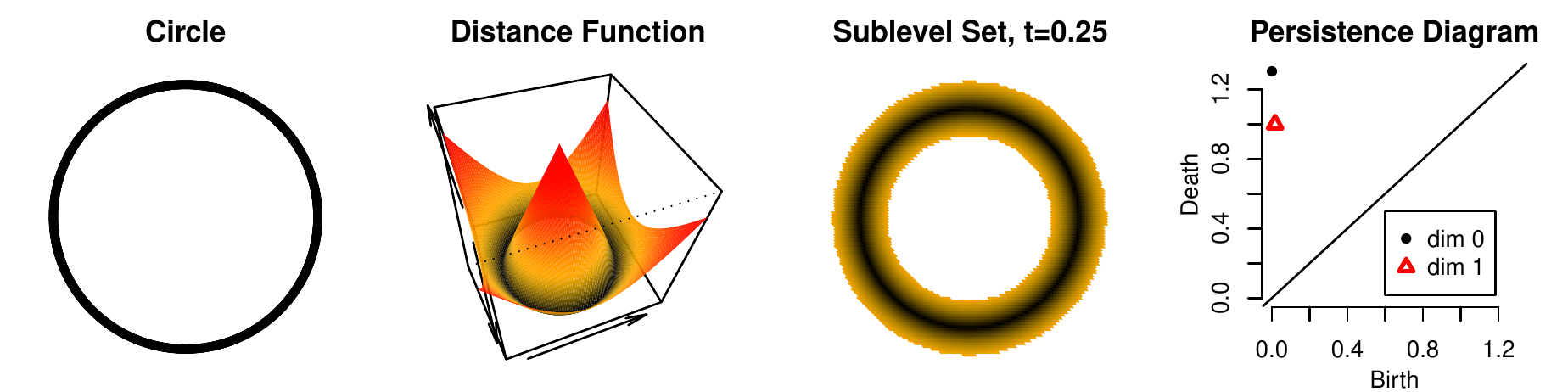}
\end{center}
\caption{The left plot shows a one-dimensional curve.
The second plot is the distance function.
The third plot shows a typical sublevel set of the distance function.
The fourth plot is the persistence diagram which shows the birth and death
times of loops (triangles) and connected components (points)
of the sublevel sets.}
\label{fig:circle}
\end{figure}

Let $S_1$ and
$S_2$ be compact sets with distance functions
$\Delta_1$ and $\Delta_2$ and diagrams
$D_1$ and $D_2$.
The bottleneck distance between
$D_1$ and $D_2$ is defined by
\begin{equation}
W_{\infty}(D_1,D_2) = \min_{g:\ D_1\to D_2}\ \sup_{z \in D_1}\Vert z-g(z) \Vert_\infty,
\end{equation}
where the minimum is over all bijections between $D_1$ and $D_2$.
In words, the bottleneck distance is the maximum distance between the points of
the two diagrams,
after minimizing over all possible pairings of the points (including the 
points on the diagonals).

A fundamental property of persistence diagrams is their \emph{stability}.
According to the Persistence Stability Theorem (\cite{steiner05stable,chazal2012structure})
\begin{equation}
W_{\infty}(D_1,D_2) \leq ||\Delta_1 - \Delta_2||_\infty = H(S_1,S_2).
\end{equation}
Here,
$H$ is the Hausdorff distance, namely,
$$
H(A,B) = \inf \Bigl\{ \epsilon:\ A \subset B \oplus \epsilon\ {\rm and}\ 
B \subset A \oplus \epsilon\Bigr\},
$$
where we recall that
$A \oplus \epsilon = \bigcup_{x\in A} B(x,\epsilon)$.
More generally, the definition of persistence diagrams and the above stability 
theorem do not restrict to distance functions but also extend to families of 
sublevel sets (resp. upper-level sets) of functions defined on $\mathbb{R}^d$ 
under very weak assumption. We refer the reader to 
\cite{edelsbrunner2010computational,ccggo-ppmtd-09,chazal2012structure} for a 
detailed exposition of the theory.

Given a sample
$X_1,\ldots, X_n\sim P$, the empirical distance function
is defined by
\begin{equation}
\hat \Delta(x) = \min_{X_i}\Vert x-X_i\Vert .
\end{equation}

\begin{lemma}[Lemma 4 in \citealp{fasy2014statistical}]
Suppose that $P$ is supported on $S$,
and has a density bounded away from zero and infinity.
Then
$$
\sup_x |\hat \Delta(x) -\Delta_S(x)|\stackrel{P}{\to} 0.
$$
\end{lemma}

See also \cite{cuevas2004boundary}.
The previous lemma justifies using $\hat \Delta$ to estimate the
persistent homology of sublevel sets of $\Delta_S$.
In fact, the sublevel sets of
$\hat \Delta$ are just unions of balls
around the observed data.
That is,
$$
L_t=\bigl\{x:\ \hat\Delta(x) \leq t\bigr\} = \bigcup_{i=1}^n B(X_i,t).
$$
The persistent homology of the union of the balls as $t$ increases
may be computed by creating a combinatorial representation 
(called a Cech complex)
of the union of 
balls, and then applying basic operations
from linear algebra \citep[Sections VI.2 and 
VII.1]{edelsbrunner2010computational}.

However,
as soon as there is noise or outliers,
the empirical distance function becomes useless, as illustrated in
Figure~\ref{fig:cassini}.
More specifically, suppose that
\begin{equation}\label{eq::model}
P = \pi R + (1-\pi) (Q\star \Phi_\sigma),
\end{equation}
where $\pi \in [0,1]$, 
$R$ is an outlier distribution (such as a uniform on a large set),
$Q$ is supported on $S$,
$\star$ denotes convolution,
and $\Phi_\sigma$ is a compactly supported noise distribution with scale parameter $\sigma$. 

Recovering the persistent homology of $\Delta_S$ exactly
(or even the homology of $S$)
is not possible in general since the problem is under-identified.
But we would still like to find a function that
is similar to the distance function for $S$.
The empirical distance function fails miserably
even when $\pi$ and $\sigma$ are small.
Instead, we now turn to the
DTM.

\begin{figure}[!ht]
\begin{center}
\includegraphics[scale=0.89]{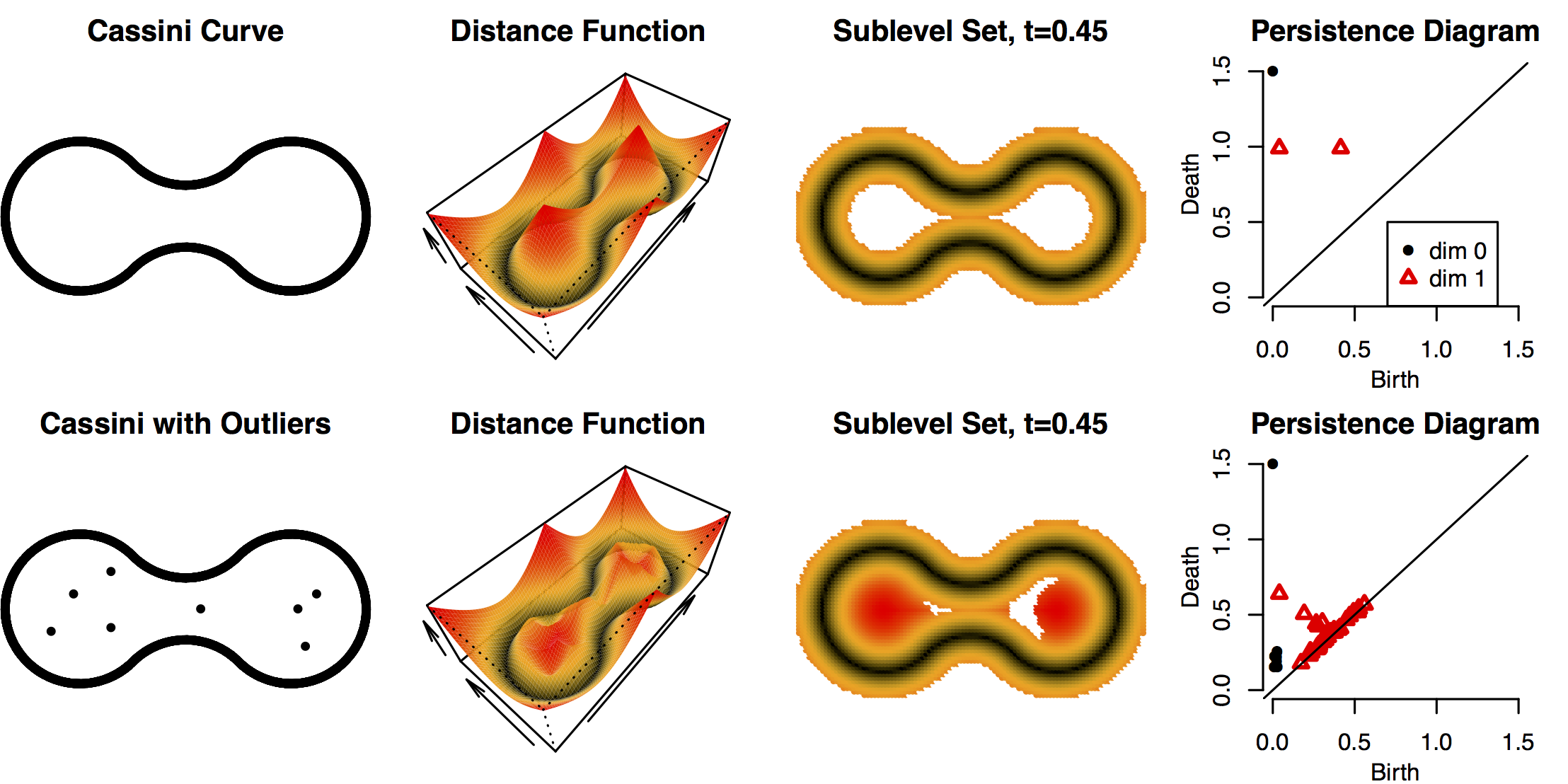}
\end{center}
\caption{Top: data on the Cassini curve, the distance function 
$\hat\Delta$,
a typical sublevel set $\{x:\ \hat\Delta(x)\leq t\}$ and the resulting 
persistence diagram.
Bottom: the effect of adding a few outliers.
Note that the distance function and persistence diagram
are dramatically different.}
\label{fig:cassini}
\end{figure}

\subsection{Distance to a Measure}

Given a probability measure $P$, 
for $0 < m < 1$,
the {\em distance-to-measure (DTM)} 
at resolution~$m$
\citep{chazal2011geometric}
is defined by
\begin{equation}
\dtm(x) \equiv \dtm_{P,m}(x)=\sqrt{ \frac{1}{m}\int_0^m ( G_x^{-1}(u))^2 du}=
\sqrt{ \mathbb{E}\left[ \Vert X-x \Vert^2\ \mathbbm{1}\Big( \Vert X-x \Vert \leq 
G_x^{-1}(m) \Big) \right] },
\label{eq:DTM}
\end{equation}
where
$G_x(t) = P( \Vert X-x \Vert \leq t)$.
Alternatively, the DTM can be defined using the cdf of the squared distances, as 
in the following lemma:
\begin{lemma}[\citealp{massart2014}]
\label{lemma::simplify}
Let $F_x(t)=P( \Vert X-x \Vert^2 \leq t)$.
Then
$$
\dtm_{P,m}^2(x) = \frac{1}{m}\int_0^m F_x^{-1}(u) du.
$$
\end{lemma}

\begin{proof}
For any $0 < u < 1$,
\begin{align*}
\left[ G_x^{-1}(u) \right]^2 &= \inf \left\{ t^2: G_x(t) \geq u \right\} =
\inf \left\{ t^2: P( \Vert X-x \Vert \leq t) \geq u \right\} \\
&= \inf \left\{ t: P( \Vert X-x \Vert^2 \leq t) \geq u \right\} =
\inf \left\{ t: F_x(t) \geq u \right\} = F_x^{-1}(u).
\end{align*}
Therefore
$$
\dtm_{P,m}^2(x) = \frac{1}{m}\int_0^m ( G_x^{-1}(u))^2 du = \frac{1}{m}\int_0^m
F_x^{-1}(u) du.
$$
\end{proof}

Given a sample
$X_1,\ldots, X_n\sim P$, 
let $P_n$ be the probability measure that puts mass $1/n$ on each~$X_i$. It is
easy to see that the distance to the measure $P_n$ 
at resolution $m$ is
\begin{equation}
\hat \dtm^2(x) \equiv \dtm_{P_n,m}^2(x) = \frac{1}{k} \sum_{X_i \in N_k(x)}
\Vert X_i-x \Vert^2,
\end{equation}
where $k= \lceil mn \rceil$ and $N_k(x)$ is the 
set containing the $k$ nearest neighbors 
of $x$ among $X_1, \ldots, X_n$.
We will use $\edtm$ to estimate $\dtm$.

Now we summarize some important properties of the DTM, all of which are proved
in \cite{chazal2011geometric}.
First, recall that the {\em Wasserstein distance of order $p$}
between two probability measures $P$ and $Q$ is given by
\begin{equation}
W_p(P,Q) = \inf_J \left(\int \Vert x-y \Vert^p \,dJ(x,y)\right)^{1/p},
\end{equation}
where the infimum is over all
joint distributions $J$ for
$(X,Y)$ such that
$X\sim P$ and $Y\sim Q$.
We say that $P$ satisfies the {\em $(a,b)$-condition} if
there exist $a,b>0$ such that,
for every $x$ in the support of $P$ and every $\epsilon>0$,
\begin{equation}\label{eq::ab}
P\Bigl(B(x,\epsilon)\Bigr)\geq a \epsilon^b.
\end{equation}

The next theorem summarizes results from \cite{chazal2011geometric} and \cite{buchet2013efficient}.
\begin{theorem}[Properties of DTM]
The following properties hold:
\begin{enumerate}
\item The distance to measure is 1-Lipschitz:  for any probability measure $P$
on $\R^d$ and any~$(x,x') \in \R^d$, 
$$ | \delta_{P,m}(x) - \delta_{P,m}(x')  \leq \| x-x'\|.$$
\item If $P$ satisfies (\ref{eq::ab}) and is supported on a compact set $S$,
then
\begin{equation}
\sup_x| \dtm_{P,m}(x) - \Delta_S(x)| \leq a^{-1/b} m^{1/b}.
\end{equation}
In particular,
$\sup_x| \dtm_{P,m}(x) - \Delta_S(x)|\to 0$ as $m\to 0$.
\item If $P$ and $Q$ are two distributions, then
\begin{equation}
\sup_x| \dtm_{P,m}(x) - \dtm_{Q,m}(x)| \leq \frac{1}{\sqrt{m}}W_2(P,Q).
\end{equation}
\item If $Q$ satisfies (\ref{eq::ab}) and is supported on a compact set $S$ 
and $P$ is another distribution (not necessarily supported on $S$),
then
\begin{equation}
\sup_x| \dtm_{P,m}(x) - \Delta_S(x)| \leq a^{-1/b} m^{1/b} + 
\frac{1}{\sqrt{m}}W_2(P,Q)
\end{equation}
Hence, if $m\asymp W_2(P,Q)^{2b/(2+b)}$, then
$\sup_x| \dtm_{Q,m}(x) - \Delta_S(x)| =O( W_2(P,Q)^{2/(2+b)})$.
\item 
Let $D_P$ be the diagram from $\dtm_{P,m}$ and let
$D_Q$ be the diagram from $\dtm_{Q,m}$.
The bottleneck distance is bounded by
\begin{equation}
W_{\infty}\left( D_{P},D_{Q} \right) \leq || \dtm_{P,m} - \dtm_{Q,m}||_\infty.
\end{equation}
\end{enumerate}
\end{theorem}

For any compact set $A\subset \mathbb{R}^d$, let $r(A)$ denotes the radius of the smallest enclosing ball of $A$
centered at zero:
$$ r (A ) = \inf \left\{ r >0  : A \subset B(0,r) \right\} .$$ 
We conclude this section by bounding the distance between the diagrams
$D_{\dtm_{P,m}}$ and $D_{\Delta_S}$.
\begin{lemma}[Comparison of Diagrams]
Let
$P = \pi R + (1-\pi) (Q\star \Phi_\sigma)$
where $Q$ is supported on $S$ and  satisfies (\ref{eq::ab}),
$R$ is uniform on a compact set $A\subset \mathbb{R}^d$ and
$\Phi_\sigma = N(0,\sigma^2 I)$.
Then,
$$
W_{\infty} \left( D_{\dtm_{P,m}},D_{\Delta_S} \right) \leq
a^{-1/b} m^{1/b} + \frac{ \pi  \sqrt{   r(A) ^2   +  2 r(S) ^2 + 2 \sigma^2     }   +  \sigma }{\sqrt{m}}.
$$
\end{lemma}
\begin{proof}
We first apply the stability theorem and part 3 in the previous result:
\begin{eqnarray*}
W_{\infty} \left( D_{\dtm_{P,m}},D_{\Delta_S} \right) &\leq & a^{-1/b} m^{1/b} + 
\frac{1}{\sqrt{m}}W_2(P,Q).
\end{eqnarray*}
The term $W_2(P,Q)$ can be upper bounded as follows:
\begin{equation*} \label{W2PQ}
W_2(P,Q) \leq W_2\left(P , Q\star\Phi_\sigma\right) + W_2\left( Q \star\Phi_\sigma, Q\right) 
\end{equation*}
These two terms can be bounded with simple transport plans. Let $Z$ be a
Bernoulli random variable with parameter $\pi$. Let $X$ and $Y$ be random
variables with distributions $R$ and $Q\star\Phi_\sigma$. We take these three
random variables independent. Then, the random variable $V$ defined by $V = Z X
+ (1-Z) Y$ has for distribution the mixture distribution $P$. By definition of
$W_2$, one has 
\begin{eqnarray*}
  W_2 ^2 \left(P  , Q\star\Phi_\sigma\right)  &\leq  & \E  \left(  \|V -  Y \|^2 \right) \\
& \leq  &   \E  \left(  | Z | ^2 \right)   \E  \left(    \| X -  Y \|^2 \right),
\end{eqnarray*}
by definition of $V$ and by independence of $Z$ and $X-Y$. Next, we have  $ \E  \left(    \| X  \|^2 \right) \leq r(A) ^2 $ and  $ \E  \left(    \| Y  \|^2 \right)  \leq 2 [r(S) ^2 + \sigma^2]$.  Thus
\begin{eqnarray*}
  W_2 ^2 \left(P  , Q\star\Phi_\sigma\right)  &\leq  &  \pi^2    \left(  r(A) ^2     +  2 r(S) ^2 + 2 \sigma^2  \right).
\end{eqnarray*}
It can be checked in a similar way that $W_2\left( Q \star\Phi_\sigma, Q\right)  \leq \sigma $ (see for instance the proof of Proposition 1 in \cite{caillerie2011deconvolution}) and the Lemma is proved.
\end{proof}

\section{Limiting Distribution of the Empirical DTM}
In this section, we find the limiting distribution of
$\edtm$ and we use this to find confidence bands for $\dtm(x)$. We start with the pointwise limit. 

Let $\dtm(x) \equiv \dtm_{P,m}(x)$ and $\edtm(x) \equiv \dtm_{P_n,m}(x)$, 
as defined in the previous section. 
\begin{theorem}[Convergence to Normal Distribution]
\label{th:pointwiseDTM}
Let $P$ be some distribution in $\R^d$. For some fixed $x$, assume that $F_x$ is differentiable at $F_x^{-1}(m)$, for $m \in
(0,1)$, with positive derivative $F_x'(F_x^{-1}(m))$.  Then we have
\begin{equation}
\sqrt{n}(\edtm^2(x) - \dtm^2(x))\rightsquigarrow N(0,\sigma^2_x),
\end{equation}
where 
$$
\sigma_x^2 = \frac{1}{m^2} \int_0^{F_x^{-1}(m)} \int_0^{F_x^{-1}(m)} 
[F_x(s \wedge t) - F_x(s) F_x(t)] \; ds \; dt.
$$
\end{theorem}
\begin{remark}
Note that assuming that $F_x$ is differentiable is not a strong assumption. According to the Lebesgue differentiation theorem on $\R$, it will be satisfied as soon as the push forward measure of $P$ by the function $\| x - \cdot\| ^2$ is absolutely continuous with respect to the Lebesgue measure on $\R$. \end{remark}




\begin{proof}
From Lemma \ref{lemma::simplify},
$$
\dtm^2(x) = \frac{1}{m}\int_0^m (G_x^{-1}(t))^2 dt = \frac{1}{m}\int_0^m 
F_x^{-1}(t) dt
$$
where
$G_x(t) = \mathbb{P}(\Vert X-x\Vert \leq t)$ and
$F_x(t) = \mathbb{P}(\Vert X-x\Vert ^2\leq t)$.
So
\begin{equation}\label{eq:integral}
\sqrt{n}(\edtm^2(x) - \dtm^2(x))=
\frac{1}{m}\int_0^m \sqrt{n}[\hat F_x^{-1}(t) - F_x^{-1}(t)] dt.
\end{equation}
First suppose that
$\hat F_x^{-1}(m) >  F_x^{-1}(m)$.
\begin{figure}[htb]
    \begin{center}
    \includegraphics[height=1.5in]{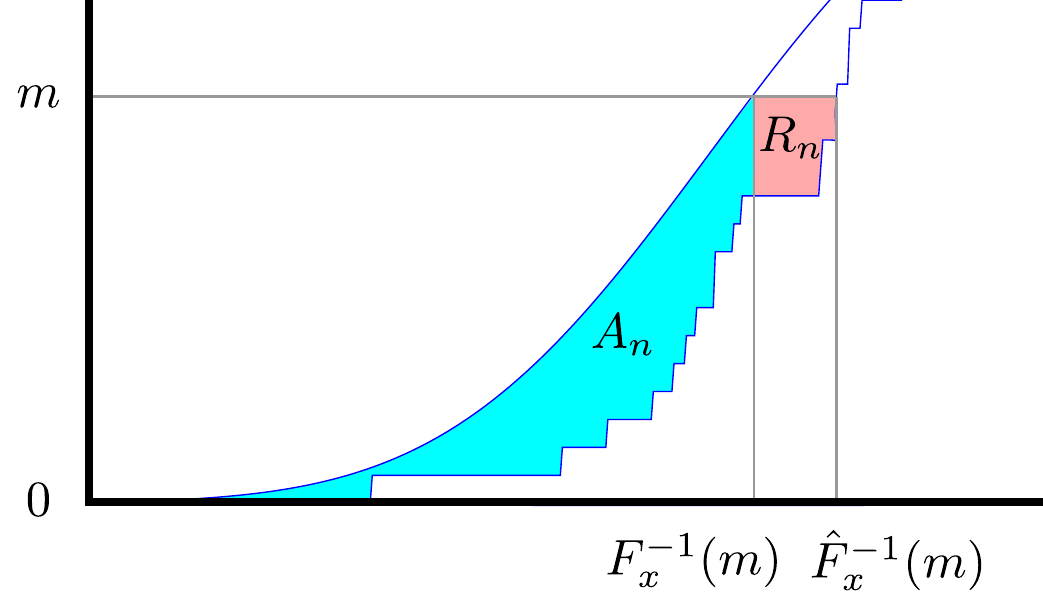}
    \end{center}
    \caption{The integral of \eqref{eq:integral} can be 
decomposed into two parts, $A_n$ and $R_n$.}
    \label{fig:cdf}
\end{figure}
Then, by integrating ``horizontally'' rather than ``vertically'',
we can split the integral into two parts, as illustrated 
in \figref{fig:cdf}:
\begin{align}
\frac{1}{m}\int_0^m \sqrt{n}[\hat F_x^{-1}(t) - F_x^{-1}(t)] dt &=
\frac{1}{m}\int_0^{F_x^{-1}(m)} \sqrt{n}[F_x(t) - \hat F_x(t)] dt +
\frac{1}{m}\int_{F_x^{-1}(m)}^{\hat F_x^{-1}(m)} \sqrt{n}[m - \hat F_x(t)] dt \notag \\
&\equiv
A_n(x) + R_n(x) \label{decompos:AnRn}
\end{align}
Next, it can be easily checked that (\ref{decompos:AnRn}) is also true
when $\hat F_x^{-1}(m) < F_x^{-1}(m)$ if we take $\int_a^b f(u) du := - \int_b^a f(u) du$ 
when $a > b$. Now, since $F_x$ is differentiable
at $m$, we have that $ \Bigl|F_x^{-1}(m) - \hat F_x^{-1}(m)\Bigr| =O_P(1/\sqrt{n})$, 
see for instance Corollary 21.5 in
\cite{van2000asymptotic}. According to the DKW inequality we have that
$\sup_t \left| F_x(t) - \hat F_x(t) \right| =O_P(\sqrt{1/n})$ and thus
$$
|R_n| \leq 
\frac{\sqrt{n}}{m}\ \Bigl|F_x^{-1}(m) - \hat F_x^{-1}(m)\Bigr|\  
\sup_t \left| F_x(t) - \hat F_x(t) \right|  = o_P(1).
$$

Next, note that $\sqrt{n}[F_x(t) - \hat F_x(t)] \rightsquigarrow 
\mathbb{B}(t)$, where $\mathbb{B}(t)$
is a Gaussian process with covariance function $[F_x(s \wedge t) - F_x(s) 
F_x(t)]$ (See, for example, \cite{van1996weak}).
By taking the integral, which is a bounded operator, 
we have that
$$
A_n \rightsquigarrow \int_0^{F_x^{-1}(m)}\mathbb{B}(t)\,dt \stackrel{d}{=} 
N(0,\sigma^2_x),
$$
where
$$
\sigma_x^2 = \frac{1}{m^2} \int_0^{F_x^{-1}(m)} \int_0^{F_x^{-1}(m)} 
[F_x(s \wedge t) - F_x(s) F_x(t)] \; ds \; dt.
$$
\end{proof}

Now, we consider the functional limit of the distance to measure, on a compact 
domain  $\mathcal X \subset \R^d$. The 
functional convergence of the DTM requires assumptions on the regularity of the quantile functions $F_x^{-1}$. 
We say that  $\omega_x : (0,1) \rightarrow \R^+ $ is a {\em modulus of continuity} of $F_x^{-1}$ if, for any $u\in (0,1)$,
\begin{equation*}
 \sup_{(m,m') \in (0,1)^2\; , \;  |m' - m|  < u} | F_x^{-1} (m') - F_x^{-1}(m)| \leq \omega_x(u),
\end{equation*}
with $  \lim_{u \rightarrow 0}  \omega_x(u) = \omega_x(0) = 0$. 
We say that $\omega_{\mathcal X}: (0,1) \rightarrow \R^+$ is an {\em uniform modulus of continuity} for the family of quantiles functions  $(F_x^{-1})_{\mathcal X}$ if, for any $u\in (0,1)$ and any $x \in \mathcal X$,
\begin{equation*}
\sup_{(m,m') \in (0,1)^2\; , \;  |m' - m|  < u} | F_x^{-1} (m') - F_x^{-1}(m)|  \leq \omega_{\mathcal X}(u),
\end{equation*}
with $ \lim_{u \rightarrow 0} \omega_{\mathcal X}(u) =
\omega_{\mathcal X}(0) = 0$.  When such modulus of continuity $\omega
$ exists, note that it always can be chosen non decreasing and this
allows us to consider its generalized inverse $\omega ^{-1}$.

One may ask if the existence of the uniform modulus of continuity over a compact 
domain $\cal X$ is a strong assumption or not.   To answer this issue, let us 
introduce the following assumption:

\medskip

\noindent $\left(H_{\omega,\mathcal X} \right):$ for any $x \in
\mathcal X$, the push forward measure $P_x$ of $P$ by $\|x-\cdot \|^2$
is supported on an interval and the absolutely continuous component of
$P_x$ has on that interval an a.e. positive density with respect to
the Lebesgue measure on $\R$.

\medskip

Note that Assumption $\left(H_{\omega,\mathcal X} \right)$ is not very strong.  
For instance it is satisfied for a measure $P$ supported on a compact and 
connected manifold, with $P_x$ absolutely continuous for the Hausdorff measure 
on $P$. The following Lemma  derives from general  results on quantile functions 
given in~\cite{bobkov2014one} (see their Appendix A);
the lemma shows that a uniform 
modulus of continuity for the quantiles exists under 
Assumption~$\left(H_{\omega,\mathcal X} \right)$.

\begin{lemma}[Existence of Uniform Modulus of Continuity]
Let $\cal X$ be a compact domain and let $P$ be a measure with compact
support in $\R^d$. Assume that Assumption~$\left(H_{\omega,\mathcal X}
\right)$ is satisfied. Then there exists an uniform modulus of
continuity for the family of quantile functions $F_x^{-1}$ over $\cal X$.
\end{lemma}

\begin{proof}
Let $x \in \mathcal{X}$. According to Proposition A.17 in \cite{bobkov2014one}, Assumption~$\left(H_{\omega,\mathcal X} \right)$ is equivalent to the absolute continuity of $F_x^{-1}$ in $[0,1)$.  
We can then define a modulus of continuity of $F_x^{-1}$ by 
$$u \in (0,1) \,  \mapsto \, \omega_x (u):= \sup_{(m,m') \in (0,1)^2\; , \;  |m' - m|  < u} \left|F_x^{-1} (m') - F_x^{-1}(m) \right| . $$
According to Lemma~\ref{lemma::lipschitz}, we have that for any $(x,x') \in \mathcal X  ^2 $:
\begin{equation}
\label{QuantileLip}
\left | F_{x'}^{-1}(m) - F_{x}^{-1}(m) \right | \leq  C   \Vert x' - x \Vert , 
\end{equation}
where $C$ only depends on $P$ and $\cal X$.  According to \eqref{QuantileLip}, for any $(m,m') \in (0,1)^2$, and for any $(x,x') \in \mathcal X  ^2 $:
 $$ \left|F_x^{-1} (m') - F_x^{-1}(m) \right| \leq   \left|F_{x'}^{-1} (m') - F_{x'}^{-1}(m) \right| + 2   C \Vert x' - x \Vert . $$ 
By taking the supremum over the $m$ and the $m'$ such that $|m' - m|  < u$, it yields:
 $$ \omega_x (u)  \leq \omega_{x'} (u)   +  2   C \Vert x' - x \Vert ,$$ 
and $x \mapsto \omega_{x}(u) $ is thus Lipschitz at any $u$.  For any $u \in (0,1)$, let
$$\omega_{\cal X}(u)  := \sup_{ x \in \mathcal X}  \omega_x (u)  ,$$
which is finite because the function $x \mapsto \omega_{x}(u) $ is continuous on the compact $\cal X$ for any $u \in
(0,1)$. We only need to prove that  $\omega_{\cal X}$ is continuous at $0$. Let $(u_n) \in (0,1)^{\mathbb N}$ be a
decreasing sequence to zero. Since   $\omega_{\cal X}$ is a non decreasing function, $\omega_{\cal X} (u_n)$ has a
limit. For any $n \in \mathbb N$, there exists a point $x_n \in \cal X$ such that $\omega_{\cal X}(u_n)  =
\omega_{x_n}(u_n)$. Let $x_{\phi(n)}$ be a subsequence which converges to $\bar x \in \mathcal X$. According to
\eqref{QuantileLip}, 
\begin{eqnarray*}
\omega_{\cal X}(u_{\phi(n)})& \leq  &\left|\omega_{x_{\phi(n)}}(u_{\phi(n)}) - \omega_{\bar x} (u_{\phi(n)})  \right| + \left|  \omega_{\bar x}(u_{\phi(n)}) \right| \\
& \leq & C \left\|x_{\phi(n)} - \bar x \right\| + \left|  \omega_{\bar x}(u_{\phi(n)}) \right|
\end{eqnarray*}
which gives that $\omega_{\cal X}(u_{\phi(n)})$ and  $\omega_{\cal X}(u_n)$ both tend to zero because $\omega_{\bar x}$ is continuous at zero. Thus $\omega_{\cal X}$ is continuous at zero and the Lemma is proved.
\end{proof}

We will also need the the following result, which shows that on any
compact domain $\mathcal X$, the function $x \mapsto F_{x}^{-1}(m)$ is
Lipschitz. For a domain $\mathcal X \in \R^d$, a probability $P$ and a
level $m$, we introduce the quantity  $q_{P,\mathcal X}(m) \in  \bar \R$,
defined by
$$q_{P,\mathcal X}(m) := \sup _{x \in \mathcal X}  F_x ^{-1} (m)    .$$
\begin{lemma}[Lipschitz Lemma]
\label{lemma::lipschitz}
Let $P$  be a measure on $\R^d$ and let $m \in (0,1)$. Then, 
for any $(x,x') \in \R^d$,
$$
\left | \sqrt{F_{x'}^{-1}(m)} - \sqrt{F_{x}^{-1}(m)} \right | \leq \Vert 
x' - x\Vert .
$$
Moreover, if  $\mathcal X$ is a compact domain in $\R^d$, then $q_{P,\mathcal
X}(m)  < \infty $ and for any $(x,x') \in \mathcal X  ^2 $:
$$  
\left | F_{x'}^{-1}(m) - F_{x}^{-1}(m) \right | \leq  2  \sqrt {q_{P,\mathcal X} (m)}  \  \Vert x' - x \Vert  . 
$$
\end{lemma}
\begin{proof}
Let $(x,a) \in \R ^2$, note that
$$
B\left(x, \sqrt{F_x^{-1}(m)} \right) \subseteq B\left(x+a, 
\sqrt{F_x^{-1}(m)} +\Vert a \Vert \right),
$$
which implies 
$$
m= \mathbb{P}\left[ B\left(x, \sqrt{F_x^{-1}(m)}\right)  \right] \leq 
\mathbb{P} \left[ B\left(x+a, \sqrt{F_x^{-1}(m)}+\Vert a \Vert 
\right) \right].
$$
Therefore $\sqrt{F_{x+a}^{-1}(m)} \leq \sqrt{F_x^{-1}(m)} + \Vert a 
\Vert$. Similarly,
$$
m = \mathbb{P}\left[ B\left(x+a, \sqrt{F_{x+a}^{-1}(m)} \right)  
\right] 
\leq \mathbb{P} \left[ B\left(x, \sqrt{F_{x+a}^{-1}(m)}+\Vert a \Vert 
\right) \right],
$$
which implies $\sqrt{F_x^{-1}(m)} \leq \sqrt{F_{x+a}^{-1}(m)} + \Vert 
a \Vert$.

Let $\mathcal X$ be a compact domain of $\R^d$, then according to the previous result for some fixed $x \in \mathcal X$ and for any $x' \in \mathcal X$, $ \sqrt {F_{x'} ^{-1}  (m) } \leq \| x' -x \|  +  \sqrt {F_{x} ^{-1}  (m)} $ which is bounded on $\mathcal X$.  The last statement follows from the fact that $|x-y| = |\sqrt{x} - \sqrt{y}|\ |\sqrt{x} + \sqrt{y}|$.
\end{proof}

We are now in position to state the functional limit of the distance to measure to the empirical measure.
\begin{theorem}[Functional Limit] \label{theo:functionalDTM}
Let $P$ be a measure on $\R^d$ with compact support.  Let  $\mathcal X$ be a compact
domain on $\R^d$ and $m \in (0,1)$. Assume that there exists a uniform modulus of continuity $\omega_{\mathcal X}$ for the family $(F_x^{-1})_{\mathcal X}$.
Then
$\sqrt{n}(\edtm^2(x) - \dtm^2(x))\rightsquigarrow \mathbb{B}(x)$
for a centered Gaussian process $\mathbb{B}(x)$
with covariance~kernel 
$$
\kappa(x,y) = \frac{1}{m^2} \int_0^{F_x^{-1}(m)} 
\int_0^{F_y^{-1}(m)} 
\Biggl(\mathbb{P}\left[ B(x, \sqrt{t}) \cap B(y, \sqrt{s}) \right] - 
F_x(t)F_y(s)\; \Biggr)ds \; dt.
$$
\end{theorem}

\begin{remark} \label{rem:localmodulus}
Note that the functional limit is valid for any value of $m \in
(0,1)$. A local version of this result could be also
proposed by considering the (local) modulii of continuity of the
quantile functions at $m$. For the sake of clarity, we prefer to give
a global version.
\end{remark}

\begin{proof}
In the proof of Theorem \ref{th:pointwiseDTM} we showed that
$\sqrt{n}(\edtm^2(x) - \dtm^2(x)) = A_n(x) + R_n(x)$
where
\begin{align*}
A_n(x) &= \frac{1}{m}\int_0^{F_x^{-1}(m)} \sqrt{n}[F_x(t) - \hat F_x(t)] dt\\
R_n(x) &=\frac{1}{m}\int_{F_x^{-1}(m)}^{\hat F_x^{-1}(m)} \sqrt{n}[m - \hat 
F_x(t)] dt.
\end{align*}
First, we show that $ \sup_{x \in \mathcal X} |R_n(x)| = o_P(1)$. Then we prove 
that $A_n(x)$ converges to a Gaussian process.

Note that 
$|R_n(x)| \leq \frac{\sqrt{n}}{m} | S_n(x) | |T_n(x)| $
where
$$
S_n(x)=\left| F_x^{-1}(m) - \hat F_x^{-1}(m) \right|,\ \ \ 
T_n(x)=\sup_t \left| F_x(t) - \hat F_x(t) \right|.
$$
Let $\xi_i \sim $Uniform (0,1), for $i=1, \dots, n$ and 
let $H_n$ be their empirical distribution function. Define $k=mn$. Then
$\hat F_x^{-1}(m) \stackrel{d}{=} F_x^{-1}(\xi_{(k)})= F_x^{-1}\left( 
H_n^{-1}(m) \right)$, where 
$\xi_{(k)}$ is the $k$th order statistic.
Thus,  for any $m>0$ and any $x \in \mathcal X$:
\begin{align}
\mathbb{P} \left( |  S_n(x)| > \epsilon\right) &=
\mathbb{P} \left( |F_x^{-1} (H^{-1}_n(m)) -  F_x^{-1}(m)| > \epsilon\right) \notag \\
& \leq \mathbb{P} \left( \omega_{\mathcal X} ( | m - H^{-1}_n(m) |  ) > \epsilon\right) \notag \\
& \leq \mathbb{P} \left( | m - H^{-1}_n(m)| > \omega_{\mathcal X}^{-1}(\epsilon)\right) \notag  \\
& \leq 2 \exp \left\{  - \frac{n \left[ \omega_{\mathcal X}^{-1}(\epsilon)  \right] ^2}{m}   
\frac{1}{1 + \frac {2 \omega_{\mathcal X}^{-1}(\epsilon)  }{3 m}}   \right\}  \label{eq:deviationSn}
\end{align}
In the last line we used inequality 1 page 453 and Point (12) of
Proposition 1 page 455 of \cite{shorack2009empirical}. Note that
$\omega_{\mathcal X}^{-1}(\epsilon) >0$ for any $\varepsilon >0$
because $\omega_{\mathcal X}$ is assumed to be continuous at zero by
definition.

Fix $\varepsilon >0$.  There exists an absolute constant $C_{\mathcal X}$ 
such that there exists an integer $N \leq C_{\mathcal X}
\varepsilon ^{-d}$ and $N$ points $(x_1,\dots,x_N)$ laying in
$\mathcal X$ such that $ \bigcup_{j= 1 \dots N} B_j \supseteq \mathcal
X$, where $B_j = B(x_j, \varepsilon)$. Now, we apply
Lemma~\ref{lemma::lipschitz} with $P$, and with $P_n$ and we find that
for any $x \in B_j$:
$$  
\left | F_{x}^{-1}(m) - F_{x_j}^{-1}(m) \right | \leq  2  \sqrt {q_{P,\mathcal X} (m)} \  \varepsilon  
\quad \textrm{ and } \quad   
\left | \hat F_{x}^{-1}(m) - \hat F_{x_j}^{-1}(m) \right | \leq  2  \sqrt {q_{P_n,\mathcal X} (m)}  \ \varepsilon .
$$
Thus, for any $x \in B_j$,
\begin{align} 
\left| F_{x}^{-1}(m) - \hat F_{x}^{-1}(m)  \right|   
& \leq  \left|  F_{x}^{-1}(m) -  F_{x_j}^{-1}(m)   \right| +  \left|  F_{x_j}^{-1}(m) -   \hat F_{x_j}^{-1}(m)    \right|  + \left|  \hat  F_{x_j}^{-1}(m) -   \hat F_{x}^{-1}(m)    \right| \notag \\
& \leq   2   \left[\sqrt {q_{P,\mathcal X} (m)}  + \sqrt {q_{P_n,\mathcal X} (m)}  \right] \,   \varepsilon   + | F_{x_j}^{-1}(m) -   \hat F_{x_j}^{-1}(m)   |   \notag \\
& \leq    C    \varepsilon   + | F_{x_j}^{-1}(m) -   \hat F_{x_j}^{-1}(m)   |  
\label{MajorFx-1}
\end{align}
where $C  $ is a positive constant which only depends on $\cal X$ and $P$. Using a union bound together with  \eqref{eq:deviationSn} , we find that 
\begin{align*}
P \left(    \sup_{x \in \mathcal X}  |S_n (x) |  >  2 C     \varepsilon \right)  
& \leq P \left(    \sup_{j = 1 \dots N }  |S_n (x_j) |  >  C   \varepsilon \right) \\
 & \leq   2 C_{\mathcal X}  \varepsilon ^{-d}  \exp \left\{  - \frac{n \left[ \omega_{\mathcal X}^{-1}(C  \epsilon)  \right] ^2}{m}   \frac 1 {1 + \frac {2 \omega_{\mathcal X}^{-1}(C  \epsilon)  }{3 m}}   \right\} .
\end{align*}
Thus, $\sup_{x \in \mathcal X}  |S_n(x) |  =
o_P(1)$. Then
\begin{align}
 \sup_{x \in \mathcal X} |T_n(x)| &= \sup_ {x \in \mathcal X} \sup_t |\hat F_x(t) - F_x(t)| \notag \\
&= \sup_ {x \in \mathcal X} \sup_t \left| \mathbb{P}_n  \left( B(x,\sqrt{t} \right) - \mathbb{P}  \left( B(x,\sqrt{t} \right) \right| \notag\\
& \leq \sup_{B \in \mathcal{B}_d} |\mathbb{P}_n(B)-\mathbb{P}(B)| =
O_P\left( {\sqrt{ \frac \dimension n}} \right)  \label{ControlTn} 
\end{align}
where ${\cal B}_d$ is the set of balls in $\mathbb{R}^d$ and we used the Vapnik-Chervonenkis theorem. Finally, we obtain
that
\begin{equation} \label{eq:RnOp}
 \sup_{x \in \mathcal X}   |R_n(x)| \leq  \frac {\sqrt n} m \sup_{x \in \mathcal X}   |S_n(x) | \sup_{x \in \mathcal X} 
|T_n(x)|  = o_P(1) .
 \end{equation}
Since  $\sup_{x \in \mathcal X}   |R_n(x)| = o_P(1)$, it only remains to prove that the process $A_n$ converges to a Gaussian process.

Now, we consider the process $A_n$ on $\mathcal X$. 
Let us denote $\nu_n := \sqrt n (P_n - P)$ the empirical process. Note that 
$$
A_n(x) = \frac 1 m  \nu_n \left ( \int_ {0 } ^{F_x ^{-1} (m)}  \mathbb 1 _{  \|x -  X\|^2 \leq t}  d t  \right) =
\frac 1 m  \nu_n \left ( f_{x} \right)  
$$
where 
$f_{x} (y) :=    \left[ F_x ^{-1} (m) -    \|x -  y\|^2   \right] \wedge 0$. 
For any $(x,x') \in \mathcal X$ and
any $y \in \R^d$, we have
\begin{align*} 
| f_{x} (y) - f_{x} (y) |  & \leq  | F_x ^{-1} (m) - F_{x'} ^{-1} (m) | +  \|x -  x'\| \left[ \| x  \| + \| x'\|  + 2 
\| y \|  \right]\\
&\leq   2 \left[  r ( \mathcal X ) +  \| y\|  + \sqrt{q_{P, \mathcal X}(m)  }  \right]    \|x -  x'\| 
\end{align*}
Since $P$ is compactly supported, then the collection of functions  $ \left( f_{x} \right)_{ x \in \mathcal
X}$ is $P$-Donsker (see for instance 19.7 in~\cite{van2000asymptotic}) and $A_n(x) \rightsquigarrow \mathbb{B}(x)$
for a centered Gaussian process $\mathbb{B}(x)$
with covariance~kernel 
\begin{align*}
\kappa(x,y) &= \text{Cov}(A_n(x), A_n(y)) = \mathbb{E}[A_n(x) A_n(y)] \\
&= \frac{1}{m^2} 
\int_0^{F_x^{-1}(m)} \int_0^{F_y^{-1}(m)} 
\mathbb{E}\left[ \left(\hat F_x(t) - F_x(t)\right) \left(\hat F_y(s) - 
F_y(s)\right) \right] \; ds \; dt \\
&= \frac{1}{m^2} \int_0^{F_x^{-1}(m)} 
\int_0^{F_y^{-1}(m)} 
\Biggl(\mathbb{P}\left[ B(x, \sqrt{t}) \cap B(y, \sqrt{s}) \right] - F_x(t)F_y(s)\Biggr)  \; 
ds \; dt.
\end{align*}
\end{proof}


\section{Hadamard Differentiability and The Bootstrap}
\label{sec:hadamard}

In this section, we use the bootstrap to get a confidence band for
$\delta$.
Define $c_\alpha$ by
$$
\mathbb{P}(\sqrt{n}||\hat\delta - \delta||_\infty > c_\alpha) = \alpha.
$$
Let $X_1^*,\ldots, X_n^*$ be a sample from
the empirical measure $P_n$ and let
$\hat\delta^*$ be the corresponding empirical DTM.
The bootstrap estimate $\hat c_\alpha$ is defined by
$$
\mathbb{P}(\sqrt{n}||\hat\delta^* - \hat\delta||_\infty > \hat c_\alpha|X_1,\ldots, X_n) = \alpha.
$$
As usual, $\hat c_\alpha$ can be approximated
by Monte Carlo.
Below we show that this bootstrap is valid.
It then follows that
$$
\mathbb{P}\left(||\delta - \hat\delta||_\infty < \frac{\hat c_\alpha}{\sqrt{n}}\right)\to 1-\alpha.
$$
A different approach to the bootstrap is considered in Section \ref{section::BB}.

To prepare for our next result, let $\mathcal{B}$ denote the class of all closed
Euclidean balls in $\mathbb{R}^d$ and let $\mathbb{B}$ denote the $P$-Brownian
bridge on $\mathcal{B}$, i.e. the centered Gaussian process on $\mathcal{B}$
with covariance function 
$\kappa(B,C) = P(B\cap C)-P(B)P(C)$, with $B,C$ in $\mathcal{B}$. We will denote
with $\mathbb{B}_x(r)$ the value of $\mathbb{B}$ at $B(x,r)$, the closed ball
centered at $x \in \mathbb{R}^d$ and with radius $r>0$.

\begin{theorem}[Bootstrap Validity]
    Let $P$ be a measure on $\mathbb{R}^d$ with compact support $S$, $m \in (0,1)$ be
    fixed and $\mathcal{X}$ be a compact domain in $\mathbb{R}^d$. Assume that  
     $F_{P,x} = F_x$ is  differentiable at
    $F_x^{-1}(m)$ and that
there exist a constant $C>0$ such that for all small $\eta \in \mathbb{R}$,
\begin{equation}\label{eq:inf.deriv}
     \sup_{x \in \cal X} \left| F_x\left( F^{-1}_x(m)  \right) -  F_x(F_x^{-1}(m)+\eta)\right| <
     \epsilon \quad \text{implies} \quad 
     \left| \eta \right|< C \epsilon.
\end{equation}
 for all $x \in \mathcal{X}$.
Then,
$\mathrm{sup}_{x \in \mathcal{X}} \sqrt{n}\left|  \left( \hat\delta^*(x)
\right)^2 - \left( \hat\delta (x) \right)^2 \right|$
converges in distribution to
\[
    \mathrm{sup}_{x \in \mathcal{X}} \left| \frac{1}{m} \int_0^{F_x^{-1}(m)}
    \mathbb{B}_x(u) du \right|
    \]
    conditionally given $X_1,X_2,\ldots$, in probability.
    \label{thm:bootstrap}
\end{theorem}

We will establish the above result using the functional delta method, which entails showing that the distance to measure function is
Hadamard differentiable at $P$. In fact, the proof further shows that  the process
\[
    x \in \mathcal{X} \mapsto \sqrt{n} \left( \delta^2(x) - \hat \delta^2 (x) \right), 
    \]
    converges weakly to the Gaussian process 
    \[
 x \in \mathcal{X} \mapsto - \frac{1}{m} \int_0^{F_x^{-1}(m) } \mathbb{B}_x(u) du.
	\]	

\begin{remark}
This result is consistent with the result established in
Theorem~\ref{theo:functionalDTM}, but in order to establish Hadamard
differentiability, we use a slightly different assumption.
Theorem~\ref{theo:functionalDTM} is proved by assuming an uniform
modulus of continuity on the quantile functions $F_x^{-1}$ whereas in
Theorem~\ref{thm:bootstrap} an uniform lower bound on the derivatives
is required. These two assumptions are consistent: they both say that
$F^{-1}_x$ is well behaved in a neighborhood of $m$ for all
$x$. However, \eqref{eq:inf.deriv} is stronger.
\end{remark}


\begin{proof}[Proof of Theorem \ref{thm:bootstrap}]
Let us first give the definition of Hadamard differentiability, for
which we refer the reader to, e.g., Section 3.9 of \cite{van1996weak}.
A map $\phi$ from a normed space $(\mathcal{D}, \| \cdot \|_{\mathcal{D}})$ to a
normed space $(\mathcal{E}, \| \cdot \|_{\mathcal{E}})$ is
Hadamard differentiable at the point $x \in \mathcal{D}$ if
there exists
a continuous linear map
$\phi_x': \mathcal{D} \to \mathcal{E}$ such that
\begin{equation}\label{eq:hadamard.def}
\left| \left|
\frac{ \phi(x + t h_t)-\phi(x)}{t}-\phi_x'(h)\right| \right|_\mathcal{E} \to 0,
\end{equation}
whenever $\| h_t  - h \|_\mathcal{D} \to 0$ as $t\to 0$.

We also recall the functional delta method
\cite[see, e.g.][Theorem 3.9.4]{van1996weak}:
suppose that $T_n$ takes values in $\mathcal{D}$,
$r_n \to \infty$,
$r_n(T_n - \theta)\rightsquigarrow T$,
and suppose that
$\phi$ is Hadamard differentiable at $\theta$.
Then
$r_n(\phi(T_n) - \phi(\theta))\rightsquigarrow \phi_\theta'(T)$.
Moreover,
by Theorem 3.9.11 of \cite{van1996weak} 
the bootstrap has the same limit.
More precisely, given $X_1,X_2,\ldots$,
we have that
$r_n(\phi(T_n^*) - \phi(T_n))$
converges conditionally in distribution to
$\phi_\theta'(T)$, in probability.
This implies the validity of the bootstrap confidence sets.

We begin our proof by defining $\mathcal{M}$ to be the space of finite, $\sigma$-finite signed measures on 
$(\mathbb{R}^d,\mathcal{B}^d)$ supported on the compact set  $S$
    and the 
mapping $\| \cdot \|_\mathcal{B} \colon \mathcal{M} \mapsto \mathbb{R}$ given
by
\[
    \| \mu \|_{\mathcal{B}} = \sup_{B \in \mathcal{B}} | \mu(B)|, \quad \mu \in
    \mathcal{M}.
    \]

    \begin{lemma}[Normed Space]
	\label{lem:norm}
	The pair $(\mathcal{M},\| \cdot \|_\mathcal{B})$ is a normed space.
    \end{lemma}

    \begin{proof}
	It is clear that $\mathcal{M}$ is closed under addition
	and scalar multiplication, and so it is a  linear space. We then need to
	show that the mapping $\| \cdot \|_\mathcal{B}$ is a norm. It is
	immediate to see that it
	absolutely homogeneous and satisfies the triangle inequality: for any
	$\mu$ and $\nu$ in $\mathcal{M}$ and $c \in \mathbb{R}$, $\| c\mu
	\|_\mathcal{B} = |c| \| \mu \|_\mathcal{B}$ and  $\| \mu +
	\nu \|_\mathcal{B} \leq \| \mu \|_\mathcal{B} + \| \nu \|_\mathcal{B}$. 
	It remains to prove that $\| \mu \|_\mathcal{B} = 0$ if and only if
$\mu$ is identically zero, i.e.	$\mu(A)
	=0$ for all Borel sets $A$. One direction is immediate: if $\| \mu
	\|_\mathcal{B} > 0$, then there exists a ball $B$ such that $\mu(B)
	\neq 0$, so that $\mu \neq 0$. For the other direction, assume that $\mu \in \mathcal{M}$ is
	such  that $\| \mu \|_\mathcal{B} = 0$. 
By the Jordan decomposition, $\mu$ can be represented as the difference of two
singular, non-negative finite measures:
$\mu = \mu_+ - \mu_{-}$. The condition $\mu(B) = 0$ for all $B \in \mathcal{B}$
is equivalent to $\mu_{+}(B) = \mu_{-}(B)$ for all $B \in \mathcal{B}$. We will show
that this further implies that  the supports of  $\mu_+$ and $\mu_{-}$, denoted
with $S_{+}$
and $S_{-}$ respectively, are both
empty, and therefore that $\mu$ is identically zero.
Indeed, recall that the support of a Borel measure $\lambda$ over a topological space $\mathbb{X}$ is the
set of points $x \in \mathbb{X}$ all of whose open
neighborhoods have positive  $\lambda$-measure. In our setting this is equivalent to the
set of points in $\mathbb{R}^d$ such that all open balls centered at those points have positive
measure, which in turn is equivalent to the set of points such that all closed balls centered at those points have positive
measure.
Therefore, using the fact that $\mu^+(B) = \mu^{-}(B)$, for all $B \in
\mathcal{B}$,
\[
    S_{+} = \left\{ x \in \mathbb{R}^d \colon \mu_{+}(B(x,r))>0, \forall r>0 \right\} = \left\{ x \in \mathbb{R}^d   \colon
	\mu_{-}(B(x,r))>0, \forall r>0\right\} = S_{-}.
    \]
    where $B(X,r) = \{ y \in \mathbb{R}^d \colon \| y - x \| \leq r \}$. 
    It then follows that $S_+$ and $S_{-}$ must be empty, for  otherwise $\mu_+$ and
    $\mu_{-}$ would be mutually singular, non-zero measures with the same support, a
    contradiction.
%
%
%
%
\end{proof}

For our purposes, instead of using $\mathcal{M}$ it will be convenient to work
with the equivalent space of the
evaluations of all $\mu \in \mathcal{M}$ over the balls $\mathcal{B}$.  Formally,
let $\ell^\infty(\mathcal{B})$ denote the normed space of bounded functions on
$\mathcal{B}$ equipped with the supremum norm. Then, by 
Lemma~\ref{lem:norm}, the mapping from $\mathcal{M}$ into
$\ell^\infty(\mathcal{B})$ given by
\begin{equation}\label{eq:MtoD}
    \mu \mapsto \left( \mu \colon \mathcal{B} \to [0,1] \right)
    \end{equation}
    is a bijection on its image, which we will denote by $\mathcal{D}$. By
    definition, the
    supremum norm on $\mathcal{D}$ is exactly the norm $\| \cdot
    \|_\mathcal{B}$, so that $\mathcal{D} \subset \ell^\infty(\mathcal{B})$ equipped with the supremum norm is a
    normed space.  With a slight abuse of notation, we will identify measures
    in $\mathcal{M}$ with the corresponding points in $\mathcal{D}$  and write $\mu \in \mathcal{D}$ to
denote the signed measure $\mu$ corresponding to the point $\{\mu(B),
    B \in \mathcal{B}\}$ in $\mathcal{D}$.

 The advantage of using the space $\mathcal{D}$ instead of $\mathcal{M}$ is that
 the  convergence of the empirical process
 $(\sqrt{n}(\P_n(B)-P(B)):\ B\in{\cal B})$ to the Brownian bridge takes place in
 $\mathcal{D}$, as required by the delta-method for the bootstrap
 \citep[see][Theorem 3.9.]{van1996weak}.

    For a signed measure $\mu$ in $\mathcal{M}$, $x \in
    \mathbb{R}^d$ and $r > 0$, we set $F_{\mu,x}(r) = \mu\left( B(x,\sqrt{r})
    \right)$. Notice if $P$ is a probability measure, then  $F_{P,x}$ is the
    c.d.f. of the univariate random variable $\|X -
    x\|^2$, with $X \sim P$. For a general $\mu \in \mathcal{M}$,
    $F_{\mu,x}$ is a cadlag function, though not monotone. For any $m \in \mathbb{R}$, $\mu$ in $\mathcal{M}$ and $x \in
    \mathbb{R}^d$, set
    \[
	F^{-1}_{\mu,x}(m) = \inf\Big\{ r > 0 \colon \mu(B(x,\sqrt{r})) \geq m
    \Big\},
	\]
    where the infimum over the empty set is define to be $\infty$. If $P$ is a
    probability measure and $m \in (0,1)$ then 	$F^{-1}_{P,x}(m)$ is just the $m$-th
    quantile of the random variable $\| X - x\|^2$, $X \sim P$.

    Fix a $m \in (0,1)$ and let $\mathcal{M}_m = \mathcal{M}_m(\mathcal{X})$ denote the subset of
    $\mathcal{M}$ consisting of all finite signed measure $\mu$ such that, there
    exists a value of $r>0$ for which 
    $\inf_{x \in \mathcal{X}} \mu\left( B(x,\sqrt{r}) \right) \geq m$. Thus, for any $\mu \in
    \mathcal{M}_m$ and $x \in \mathcal{X}$, $F^{-1}_{\mu,x}(m) < \infty$. Let
    $\mathcal{D}_m$ be the image of $\mathcal{M}_m$ by the mapping
    \eqref{eq:MtoD}.

    Let $\mathcal{E}$ the set of
    bounded, real-valued
    function  on $\mathcal{X}$, a normed space with respect to the sup norm.
    Finally, we define $\phi \colon \mathcal{D}_m \rightarrow \mathcal{E}$ to be
    the mapping 
    \begin{equation}\label{eq:dtm.map}
	\mu \in \mathcal{D}_m \mapsto \phi(\mu)(x) =  F^{-1}_{\mu,x}(m) - \frac 1 m
	\int_{0}^{F^{-1}_{\mu,x}(m)} F_{\mu,x}(u) du, \quad x \in \mathcal{X} 
    \end{equation}
Notice that if $P$ is a probability measure, simple algebra shows that $\phi(P)(x)$ is the square value
of the distance to
measure of $P$ at the point $x$, i.e. $\delta^2_p(x)$; see 
Figure~\ref{fig:cdfHadam}.

Below we will show that, for any probability measure $P$, the mapping \eqref{eq:dtm.map} is
Hadamard differentiable at $P$.

For an arbitrary $Q \in \mathcal{Q}$, let $\{ Q_t \}_{t>0} \subset \mathcal{D} $
be a sequence of signed measure such that $\lim_{t \rightarrow 0} \| Q_t -
Q\|_{\mathcal{B}} = 0$ and such that $P + t Q_t \in \mathcal{D}_m$ for all $t$.
Sequences of this form exist: since $\| t Q_t \|_{\mathcal{B}}
\rightarrow 0$ as $t \rightarrow 0$, for any arbitrary $0 <
\epsilon < 1 - m$ and all $t$ small enough,
\[
    \inf_{x \in \mathcal{X}} \left(P + t Q_t \right)\left( B\left( x,F_{P,x}^{-1}(m+\epsilon) \right) \right)
\geq m + \epsilon/2.
    \]
    By the boundedness of $\mathcal{X}$ and compactness of $S$,
this implies that there exists a
    number $r > 0$ such that 
\[
\inf_x \left(P + t Q_t \right)\left( B\left( x,r \right)
\right) \geq m,
    \]
so the image of $P+tQ_t$ by \eqref{eq:dtm.map} is an element of
$\mathcal{E}$ (i.e. it is a bounded function).


\begin{figure}[htb]
    \begin{center}
    \includegraphics[height=1.5in]{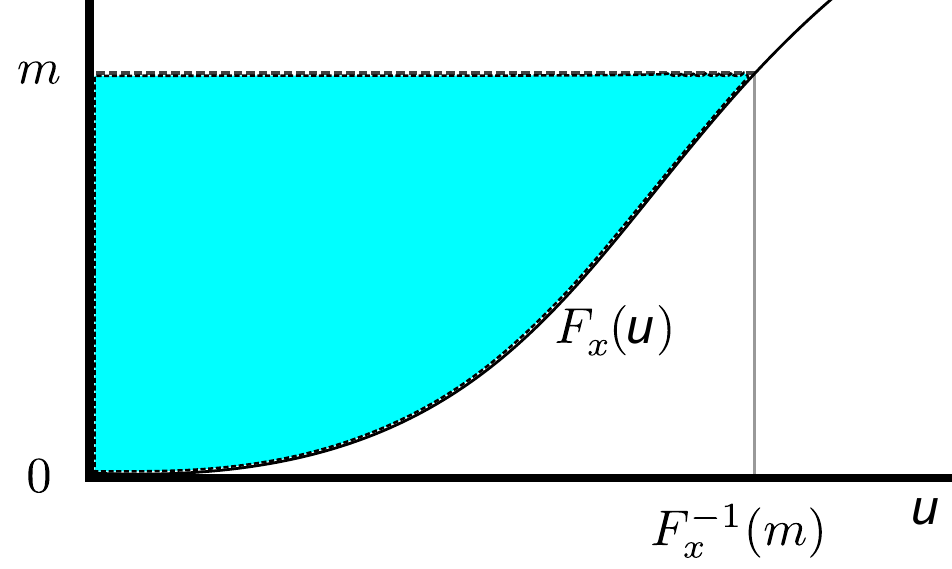}
    \end{center}
    \caption{The integral $\int_0^m F^{-1}_{P,x}(u)du$ is equivalent to $m
    F_{P,x}^{-1}(m) - \int_0^{F_{P,x}^{-1}(m)} F_{P,x}(u) du$.}
    \label{fig:cdfHadam}
\end{figure}

For sake of readability, below we will write $F_{x,t}$ and $F^{-1}_{x,t}(m)$ for $F_{P + t Q_t,x}$ and
$F^{-1}_{P+tQ_t,x}(m)$, respectively, and $F_x$ for $F_{P,x}$. 
Also, for each $x \in \mathcal{X}$ and $z \in \mathbb{R}_+$ 
we 
the set $\mathcal{A}_{x,z}=\{ y: \Vert y-x \Vert^2 \leq z \}$ and
$F_{x,t}(z)=(P+tQ_t)(\mathcal{A}_{x,z})$.

Thus, 
\[
\phi(P)(x)  =\delta^2_P(x)= F_x^{-1}(m) - \frac{1}{m} \int_0^{F_x^{-1}(m)} F_x(u) du
\]
    and
    \[
\phi(P+t Q_t)(x) = F_{x,t}^{-1}(m) - \frac{1}{m} \int_0^{F_{x,t}^{-1}(m)}
F_{x,t}(u) du.
\]
Some algebra show that, for any $x$,
\begin{equation}
\frac{\phi(P)(x)- \phi(P+tQ_t)(x)}{t} = \frac{F_x^{-1}(m) - F_{x,t}^{-1}(m)}{t}
- \frac{A(x,t)}{mt},
\label{eq:hadamard}
\end{equation}
where
$$
A(x,t) =  \left\{
    \begin{array}{cc}
 \int_0^{F_x^{-1}(m)} \left[ F_x(u) - F_{x,t}(u) \right] du  -
   \int_{F_x^{-1}(m)}^{F_{x,t}^{-1}(m)} F_{x,t}(u) du & \text{ if } F_{x}^{-1}(m) \leq F_{x,t}^{-1}(m)\\
 \int_0^{F_x^{-1}(m)} \left[ F_x(u) - F_{x,t}(u) \right] du  +
 \int_{F_{x,t}^{-1}(m)}^{F_{x}^{-1}(m)} F_{x,t}(u) du & \text{ if }
    F_{x}^{-1}(m) > F_{x,t}^{-1}(m).\\
\end{array}
\right.
$$
To demonstrate Hadamard differentiability (see \ref{eq:hadamard.def}), we will prove that, as $t \to 0$, the
expression in \eqref{eq:hadamard}, as a bounded function of $x \in \mathcal{X}$, will
converge in $\mathcal{E}$ to the bounded function
\[
    x \in \mathcal{X} \mapsto - \frac{1}{m} \int_0^{F_x^{-1}(m)} Q(\mathcal{A}_{x,u}) du.
    \]

Towards that end, we have, for all $t$ and any $x \in \mathcal{X}$,
\begin{align*}
\frac{A(x,t)}{t} &= \frac{1}{t} \left[ \int_0^{F_x^{-1}(m)} t Q_t(\mathcal{A}_{x,u}) du - \int_{F_x^{-1}(m)}^{F_{x,t}^{-1}(m)} (P+tQ_t)(\mathcal{A}_{x,u}) du \right] \\
&= \int_0^{F_x^{-1}(m)} Q_t(\mathcal{A}_{x,u}) du - \frac{1}{t} \int_{F_x^{-1}(m)}^{F_{x,t}^{-1}(m)} P(\mathcal{A}_{x,u}) du - \int_{F_x^{-1}(m)}^{F_{x,t}^{-1}(m)} (Q_t)(\mathcal{A}_{x,u}) du \\
& \equiv A_1(x,t) - A_2(x,t) -A_3(x,t),
\end{align*}
where, for $a < b$, we write $\int_b^a = -\int_a^b$.

To handle the three terms appearing in the last display, we first state and
prove two useful results.

\begin{lemma}\label{lem:AS}
Under the assumptions of the theorem,
\begin{equation}\label{eq:AS}
  \sup_{x \in \mathcal{X}} \frac{F_x^{-1}(m) -
    F_{x,t}^{-1}(m)}{t} = O(1),
	\end{equation}
	as $t \rightarrow 0$.
In particular,
\begin{equation}\label{eq:m.conv}
    \lim_{t \rightarrow 0}    \sup_{x \in \mathcal{X}} \left| m -  F_x\left( F_{x,t}^{-1}(m) \right)
\right| = 0.
    \end{equation}
and
\begin{equation}\label{eq:vanish}
    \lim_{t \rightarrow 0} \sup_{x \in \mathcal{X}} |F_{x}^{-1}(m) -
F_{x,t}^{-1}(m)| = 0.
\end{equation}
\end{lemma}

\begin{proof}
We first prove \eqref{eq:vanish}. 
Let $\mathcal{A}_{x,t} = \{ y \colon \| y - x \|^2 \leq F^{-1}_{x,t}(m)\}$. Then, 
\begin{equation}\label{eq:basic}
m = (P + t Q_t)(\mathcal{A}_{x,t} ) = F_x\left(  F^{-1}_{x,t}(m)
 \right) +  t Q_t (\mathcal{A}_{x,t}).
 \end{equation}
Since 
$\sup_{x \in \mathcal{X}} t Q_t (\mathcal{A}_{x,t}) \rightarrow 0$ as $t \rightarrow 0$ we obtain \eqref{eq:m.conv}.
 The claim \eqref{eq:vanish} follows from \eqref{eq:m.conv} using the facts that $F_x$ is monotone for each $x \in
\mathcal{X}$
and that $\inf_{x \in \mathcal{X}} F_x'(F^{-1}_x(m))  >  0$. 

To show \eqref{eq:AS}, the relation \eqref{eq:basic} combined with the fact that
$m = F_x\left( F^{-1}_x(m) \right)$ for all $x$
 yields that
	\[
	    \Big| F_x\left( F^{-1}_x(m) \right) - F_x\left(  F^{-1}_{x,t}(m)
 \right) \Big| =  t |Q_t (\mathcal{A}_{x,t})|,
	    \]
	    for all $x \in \mathcal{X}$. By \eqref{eq:vanish}
 and \eqref{eq:inf.deriv},
	    \[
		\left|  F^{-1}_x(m) -  F^{-1}_{x,t}(m) \right|  \leq C t |Q_t(S)|
		\]
		uniformly in $x \in \mathcal{X}$ and for all $t$ small enough. The relation
	\eqref{eq:AS} follows from the fact that $|Q_t(S)| = |Q(S)| + o(1) <
	\infty$.
\end{proof}

We now analyze the terms $ A_1(x,t)$, $ A_2(x,t)$ and $ A_3(x,t)$ separately.

\begin{itemize}
    \item {\bf Term $ A_1(x,t)$.}
	As $t \rightarrow 0$,  $Q_t \rightarrow Q$ and, uniformly in $x \in \mathcal{X}$ and $z >
	0$,
$|Q_t(\mathcal{A}_{x,z})| \leq |Q_t(S)| = |Q(S)| + o(1) < \infty $. Furthermore,
$\sup_{x \in \mathcal{X}} F^{-1}_x(m) < \infty$ by compactness of $\mathcal{X}$
and $S$. Therefore, using the dominated convergence
theorem,
\begin{equation}
    \lim_{t \to 0} \sup_{x \in \mathcal{X}} \left| \frac{A_1(x,t)}{m} - \frac{1}{m}\int_0^{F_x^{-1}(m)}
    Q(\mathcal{A}_{x,u}) du \right| = 0.
\label{eq:A1}
\end{equation}
\item {\bf Term $A_2(x,t)$.} Since  $P(\mathcal{A}_{x,u})$ is non-decreasing in $u$ for
    all $x$, we have
    \[
	\frac{F_{x,t}^{-1}(m) - F_{x}^{-1}(m)}{t} \times \min \left\{ m,  F_x(
     F^{-1}_{x,t}(m)) \right\} \leq A_2(x,t) \leq
 \frac{F_{x,t}^{-1}(m) - F_{x}^{-1}(m)}{t} \times \max \left\{ m, F_x(
     F^{-1}_{x,t}(m))\right\}.
	\]

%

	Using \eqref{eq:m.conv},
	we conclude that
\begin{equation}\label{eq:A2}
    \lim_{t \rightarrow 0} \sup_{x \in \mathcal{X}} \Big| \frac{F_x^{-1}(m) - F_{x,t}^{-1}(m)}{t}
 - \frac{A_2(x,t)}{m}  \Big| = 0.
 \end{equation}

 \item {\bf Term $A_3(x,t)$.} Finally, since  $|Q_t(S)| \leq |Q(S)| + o(1)$ as $t \rightarrow 0$ and using \eqref{eq:vanish},
 we obtain
 \begin{equation}
 \sup_x|  A_3(x,t) | = O \Big(  
 \sup_x \left| F_{x,t}^{-1}(m) - F_{x}^{-1}(m) \right| \Big) = o(1)
\label{eq:A3}
\end{equation}
as $t \rightarrow 0$. 

\end{itemize}
Therefore, from \eqref{eq:hadamard}, \eqref{eq:A1}, \eqref{eq:A2}, and \eqref{eq:A3},
\[
    \lim_{t \rightarrow 0} \sup_x \Big| \frac{\phi(P)(x)- \phi(P+tQ_t)(x)}{t}
 + \frac{1}{m} \int_0^{F_x^{-1}(m)} Q(\mathcal{A}_{x,u}) du   \Big| = 0,
    \]
which shows that 
\[
    x \in \mathcal{X} \mapsto - \frac{1}{m} \int_0^{F_x^{-1}(m)} Q(\mathcal{A}_{x,u}) du
    \]
is the Hadamard derivative of $\delta^2$ at $P$.

The statement of the theorem now follows from an application of Theorem 3.9.11 in
\cite{van1996weak} and the fact that, since ${\cal B}$ is a Donsker class,
the empirical process
$(\sqrt{n}(\P_n(B)-P(B)):\ B\in{\cal B})$
converges to the Brownian bridge $\mathbb{B}$ on $\mathcal{B}$
with covariance kernel
$\kappa(B,C) = P(B\cap C)-P(B)P(C)$.
\end{proof}

\subsection{Significance of Topological Features}\label{ss:significance}

Fasy et al (2014) showed how to use
the bootstrap to test the significance
of a topological feature.
They did this for distance functions and density
estimators but the same idea works for DTM as we now explain.

Given a feature
with birth and death time $(u,v)$,
we will say that the feature is significant if
$| v-u | > 2c_\alpha/\sqrt{n}$
where
$c_\alpha$ is defined by
$$
\mathbb{P}(\sqrt{n} \Vert \edtm (x) - \dtm(x) \Vert_\infty > c_\alpha) = \alpha.
$$
In particular, $c_\alpha$ can be estimated
from the bootstrap as we showed in the previous section.
Specifically,
define $\hat c_\alpha$ by
$$
\mathbb{P}(\sqrt{n} \Vert \edtm^* (x) - \edtm(x) \Vert > \hat c_\alpha|X_1,\ldots, X_n) = \alpha.
$$
Then 
$\hat c_\alpha$ is a consistent estimate of $c_\alpha$.

To see why this makes sense,
let ${\cal D}$ be the set of persistence diagrams.
Let $D \equiv D_{\dtm}$ be the true diagram and let $\hat D \equiv 
D_{\edtm}$
be the estimated diagram.
Let
$$
{\cal C}_n = \left\{E\in {\cal D}:\ W_{\infty}(\hat D,E)\leq 
\frac{\hat c_\alpha}{\sqrt{n}}\right\}.
$$
Then
$$
\mathbb{P}(D \in {\cal C}_n) = \mathbb{P}\left( W_{\infty}(D, \hat D) \leq 
\frac{\hat c_\alpha}{\sqrt{n}} \right)  
\geq \mathbb{P}(\sqrt{n} \Vert \edtm (x) - \dtm(x) \Vert \leq \hat c_\alpha) \to 
1-\alpha
$$
as $n\to \infty$.
Now
$|v-u| > 2\hat c_\alpha/\sqrt{n}$
if and only if
the feature cannot be matched to the diagonal
for any diagram in ${\cal C}$.
(Recall that the diagonal corresponds to features with zero lifetime.)

We can visualize the significant features
by putting a band of size
$2c_\alpha/\sqrt{n}$ around the diagonal of $\hat D$.
See Figure~\ref{fig:DTMcassini}.

\begin{figure}[!ht]
\begin{center}
\includegraphics[scale=0.89]{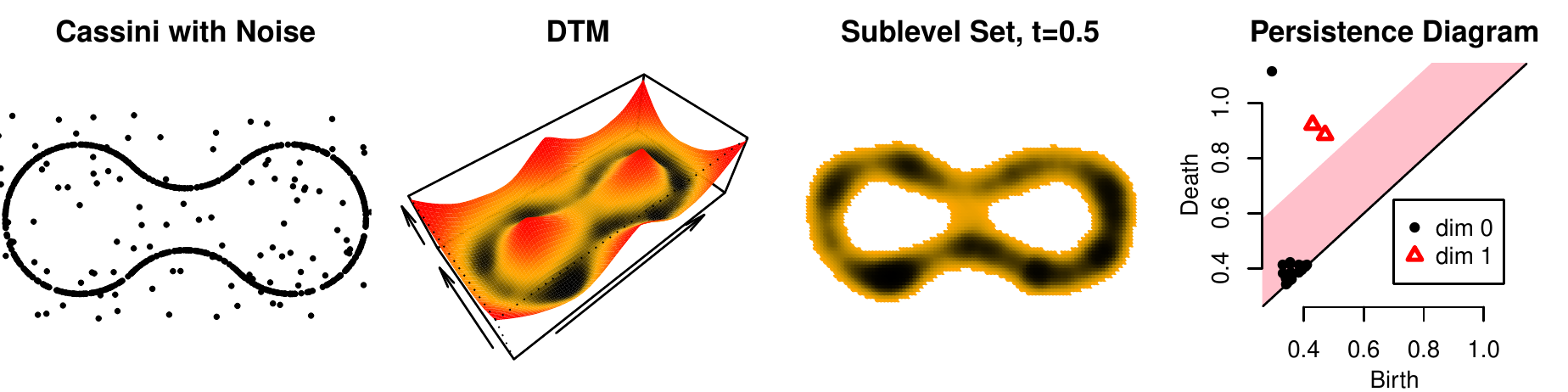}
\end{center}
\caption{The left plot shows a sample from the Cassini curve together
with a few outliers.
The second plot is the empirical DTM.
The third plot is one sub-level set of the DTM.
The last plot is the persistence diagram.
Points not in the shaded band are significant features.
Thus, this method detects one significant connected component and two
significant loops in the sublevel set filtration of the empirical DTM function.}
\label{fig:DTMcassini}
\end{figure}

\section{Theory for Kernels}
\label{section::kernels}

In this section, we consider an alternative
to the DTM, namely, kernel based methods.
This includes the kernel distance and the kernel
density estimator.


\cite{phillips2014goemetric}
suggest using the kernel distance
for topological inference.
Given a kernel $K(x,y)$, the 
kernel distance between two probability measures
$P$ and $Q$ is
$$
D_K(P,Q) = \sqrt{ \iint K(x,y)dP(x) dP(y) +
\iint K(x,y)dQ(x) dQ(y) - 2\iint K(x,y)dP(x) dQ(y) } .
$$
It can be shown that
$D_K(P,Q) = \Vert \mu_P - \mu_Q \Vert$
for vectors
$\mu_P$ and $\mu_Q$
in an appropriate reproducing kernel Hilbert space (RKHS).
Such distances are popular
in machine learning;
see \cite{sriperumbudur2009kernel}, for example.

Given a sample
$X_1,\ldots, X_n\sim P$, 
let $P_n$ be the probability measure that put mass $1/n$ on each~$X_i$.
Let $\vartheta_x$ be the Dirac measure that puts mass one on $x$.
\cite{phillips2014goemetric}
suggest using the discrete kernel distance
\begin{equation}
\hat D_K(x) \equiv D_K(P_n, \vartheta_x) = 
\sqrt{\frac{1}{n^2}\sum_{i=1}^n \sum_{j=1}^n K(X_i,X_j) +  K(x,x) -  
\frac{2}{n}\sum_{i=1}^n K(x,X_i)}\\
\end{equation}
for topological inference.
This is an estimate of the population quantity
$$
D_K(x) \equiv D_K(P, \vartheta_x) = 
\sqrt{\iint K(z,y) dP(z)dP(y) +  K(x,x) - 2 \int K(x,y) dP(y) }.
$$

The most common choice of kernel
is the Gaussian kernel
$K(x,y) \equiv K_h(x,y) = \exp \left(- \frac{\Vert x-y \Vert^2}{2h^2} \right)$,
which has one tuning parameter $h$.
We recall that, in topological inference,
we generally do not let $h$ tend to zero. See the related discussion in Section 4.4 of \cite{fasy2014statistical}.

Recall that the kernel density estimator is defined by
$$
\hat p_h(x) = \frac{1}{n (\sqrt{2 \pi} h)^{\dimension}} \sum_i K(x,X_i).
$$
We see that
\begin{align*}
\hat D_k^2(x) 
&=h^{\dimension}\left(\frac{(\sqrt{2\pi})^{\dimension}}{n}\sum_i \hat p_h(X_i) + 
h^{-{\dimension}}K(0,0) - 2 (\sqrt{2\pi})^{\dimension} \hat p_h(x)\right) \\
&=h^{\dimension}\left(\frac{(\sqrt{2\pi})^{\dimension}}{n}\sum_i \left[\hat 
p_h(X_i)-p(X_i)\right] + \frac{(\sqrt{2\pi})^{\dimension}}{n}\sum_i p(X_i)
+ h^{-{\dimension}}K(0,0) - 2 (\sqrt{2\pi})^{\dimension} \hat p_h(x)\right) \\
&= (\sqrt{2\pi})^{\dimension} h^{\dimension}(1+o_P(1)) + 
O_P \left(\sqrt{\frac{\log n}{n}} \right) + K(0,0) - 2 (\sqrt{2\pi} h)^{\dimension} \hat p_h(x).
\end{align*}
Here, we used the fact that
$n^{-1}\sum_{i=1}^n p(X_i) = 1+ o_P(1)$ and
$||\hat p_h - p||_\infty = O_P(\sqrt{\log n/n})$.

We see that up to small order terms,
the sublevel sets of
$D_K(x)$ are a rescaled version of the super-level sets of
the kernel density estimator.
Hence, the kernel distance approach and the density estimator
approach are essentially the same, up to a rescaling.
However,
$D^2_K$ has some nice properties;
see \cite{phillips2014goemetric}.

The limiting properties of
$\hat D_K^2(x)$ follow
immediately from
well-known properties of kernel density estimators.
In fact, the conditions needed for
$\hat D_K^2$ are weaker than for the DTM.

\begin{theorem}[Limiting Behavior of Kernel Distance]
We have that
$$
\sqrt{n}(\hat D_K^2 - D_K^2)\rightsquigarrow \mathbb{B},
$$
where $\mathbb{B}$ is a Brownian bridge.
The bootstrap version converges to the same limit,
conditionally almost surely.
\end{theorem}

The proof of the above theorem is based on 
existing theory and so is omitted.
This theorem justifies using the bootstrap
to construct $L_\infty$ bands for $\mathbb{E}(\hat p)$
or $D_K$.

As we mentioned before,
for topological inference, we keep the bandwidth $h$ fixed.
Thus, it is important to keep in mind that we view
$\hat p_h$ as an estimate of
$p_h(x) = \mathbb{E}[\hat p_h(x)] = \int K_h(x,u) dP(u)$.

\section{The Bottleneck Bootstrap}
\label{section::BB}

More precise inferences can be
obtained by directly bootstrapping the persistence diagram.
Define $\hat t_\alpha$ by
\begin{equation}
\mathbb{P}(\sqrt{n}W_\infty(\hat D^*, \hat D) > \hat t_\alpha \, |\, X_1,\ldots, X_n) = \alpha.
\end{equation}
The quantile 
$\hat t_\alpha$ can be estimated by Monte Carlo.
We then use a band of size $2\hat t_\alpha$
on the diagram $D$.

In the following, we show that $\hat t_\alpha$ consistently
estimates the population value $t_\alpha$ defined by
\begin{equation}
\mathbb{P}(\sqrt{n}W_\infty(\hat D, D) >  t_\alpha ) = \alpha.
\end{equation}

The reason why the bottleneck bootstrap can lead to more precise
inferences than the functional bootstrap from the previous section
is that the functional bootstrap uses the fact that
$W_\infty(\hat D,D) \leq ||\hat\delta - \delta||_\infty$ and finds an upper
bound for $||\hat\delta - \delta||_\infty$.
But in many cases the inequality is not sharp
so the confidence set can be very conservative.
Moreover, we can obtain different critical values
for different dimensions (connected components, loops, voids, ...)
and so the inferences are tuned to the specific features we are estimating.
See Figure~\ref{fig:DTMcassiniBoot}.

\begin{figure}[!ht]
\begin{center}
\includegraphics[scale=0.89]{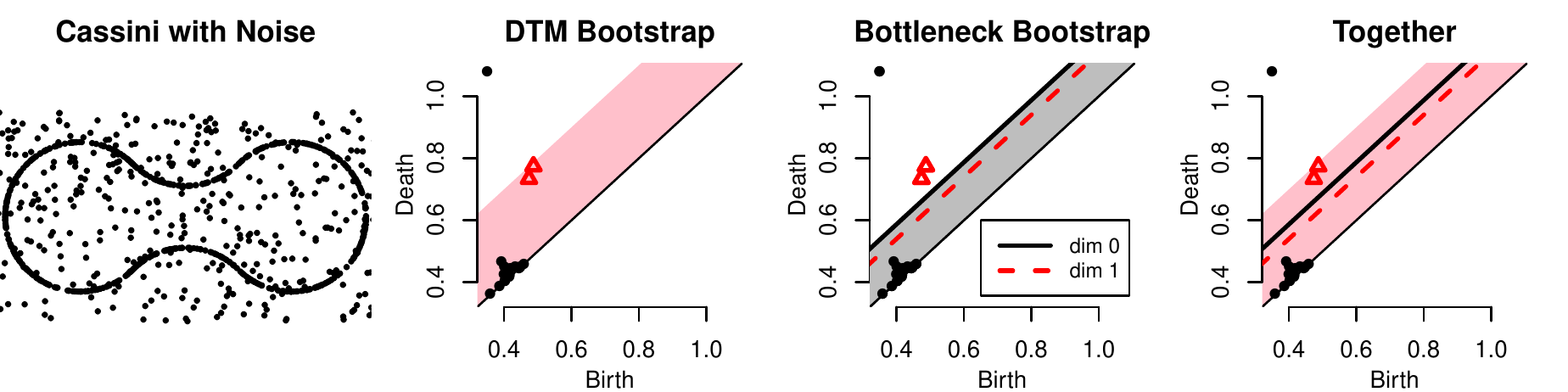}
\end{center}
\caption{The left plot shows a sample from the Cassini curve together
  with a few outliers.  The second plot shows the DTM persistence
  diagram with a 95\% confidence band constructed using the method of
  Section \ref{sec:hadamard}. The third plot shows the same
  persistence diagram with two 95\% confidence bands constructed using
  the bottleneck bootstrap with zero-dimensional features and
  one-dimensional features. The fourth plot shows the three confidence
  bands at the same time. In Section \ref{sec:examples}, we use
  this compact form to show multiple confidence bands.}
\label{fig:DTMcassiniBoot}
\end{figure}

Although the bottleneck bootstrap can be used with either the DTM or the KDE,
we shall only prove its validity for the KDE.
First, we need the following result.
For any function~$p$,
let $g = \nabla p$ denote its gradient and let
$H = \nabla^2 p$ denotes its Hessian.
We say that $x$ is a critical point if $g(x) = (0,\ldots, 0)^T$.
We then call $p(x)$ a critical value.
A function is Morse if the Hessian is non-degenerate at each critical point.
The More index of a critical point $x$ is the number of negative eigenvalues 
of~$H(x)$.


\begin{lemma}[Stability of Critical Points]
\label{lemma::stability}
Let $p$ be a density with compact support $S$.  
Assume that $S$ is a $d$-dimensional
compact submanifold of $\R^d$ with boundary.  
Assume $p$ is a Morse
function with finitely many, distinct, critical values with
corresponding critical points $c_1,\ldots, c_k$.  Also assume that $p$
is of class $\mathcal{C}^2$ on the interior of $S$, continuous and
differentiable with non vanishing gradient on the boundary of $S$.
There exists $\epsilon_0 >0$ such that for all $0< \epsilon <
\epsilon_0$ the following is true: for some positive constant $c$, there exists
$\eta \geq c \epsilon$
 such that, for any density $q$ with
support $S$ satisfying
$$
\sup_x |p(x) - q(x) | < \eta,\ \ \ \ \
\sup_x |\nabla p(x) - \nabla q(x) | < \eta,\ \ \ \ \
\sup_x |\nabla^2 p(x) - \nabla^2 q(x) | < \eta,
$$
$q$ is a Morse function with exactly $k$ critical points 
$c_1',\ldots, c_k'$ say,
and, after  a suitable re-labeling of indices,
$$
\max_j ||c_j - c_j'||\leq  \epsilon.
$$
Moreover,
$c_j$ and $c_j'$ have the same Morse index.
\end{lemma}

\begin{proof}
This lemma is a consequence of classical stability properties of Morse
functions.  First, from Theorem 5.31, p.140 in \cite{banyaga2004} and
Proposition II.2.2, p.79 in \cite{golubitsky1986}, there exists
$\epsilon_1 >0$ such that if $q$ is at distance less than $\epsilon_1$
in the $\mathcal{C}^2$ topology (i.e. such that the sup-norm of $p-q$
and its first and second derivatives are bounded by $\epsilon_1$) then
$q$ is a Morse function.  Moreover, there exist two diffeomorphisms
$h: \R \to \R$ and $\phi : S \to S$ such that $q = h \circ p \circ
\phi$. As the notion of critical point and of index are invariant by
diffeomorphism, $p$ and $q$ have the same number of critical points
with same index. More precisely, the critical points of $q$ are the
points $c_i' = \phi^{-1}(c_i)$.

Now let $\epsilon >0$ be small enough such that 
$2 \epsilon < \min_{i   \not = j} \|c_i - c_j\|$, and for any $i \not = j$,
$p(B(c_i,\epsilon)) \cap p(B(c_j,\epsilon)) = \emptyset$. Then 
$\eta_1 = \eta_1(\epsilon) = \min_{i \not = j } d( p(B(c_i,\epsilon)), p(B(c_j,\epsilon)))$ where
$d(A,B) = \min_{a \in A, b \in B} |a-b|$ and 
$\eta_2 = \eta_2(\epsilon) = \inf \{ \|\nabla p(x) \| : x \in S \setminus \cup_{i=1}^k B(c_i, \epsilon) \}$
are both positive. If $q$ satisfies the assumptions of the lemma for any 
$0 < \eta \leq \min( \eta_1, \eta_2 )$, then the critical values of $q$
have to be in $\cup_i p(B(c_i,\epsilon))$ and the critical points
$c_i'$ have to be in $\cup_i B(c_i,\epsilon)$.

More precisely, notice that since $p$ is a Morse function, for $\epsilon$ small enough, $\eta_2 = O(\epsilon)$, and, 
for any $i \in \{1, \cdots, k \}$, the Taylor series of $\nabla p$ about $c_j$ yields
\[
\nabla p(x) = H_i (x - c_i) + \| x - c_i \| r\left(  x - c_i  \right),
\]
where $r(z) \rightarrow 0$ as $\|z\| \rightarrow 0$ and $H_i$ is the Hessian of
$ p$ at $c_i$. 
Let $\lambda_{\min}$ be the smallest absolute eigenvalue of the Hessians at all the critical points. Since $p$ is a Morse function, the matrix $H_i$ is full rank and $\lambda_{\min}$ is positive.
As a consequence, 
for all $x \in S \setminus \cup_{i=1}^{k} B(c_i, \epsilon)$ and $\epsilon$ small enough,  $\| \nabla p(x) \| \geq \frac{\lambda_{\min}}{2} \epsilon$.
Since $\eta_1$ is a non-incresing function of $\epsilon$, we have that, for $\epsilon$ small enough, $\eta = \eta_2 \geq \frac{\lambda_{\min}}{2} \epsilon$.


To conclude the proof of the lemma, we need to prove that each
ball $B(c_i,\epsilon)$ contains exactly one critical point of $q$.
Indeed, for $t \in [0,1]$, the functions $q_t(x) = p(x) +
t(q(x)-p(x))$ are Morse functions satisfying the same properties as
$q$. Now, since each $c_i$ is a non-degenerate point of $p$, it
follows from the continuity of the critical points (see,
e.g. Prop. 4.6.1 in \cite{demazure2013}) that, restricting $\epsilon$
if necessary, there exist smooth functions 
$c_i : [0,1] \to S$, $c_i(0) = c_i, c_i(1)=c_i'$ such that $c_i(t)$ is the unique critical
point of $q_t$ in $B(c_i,\epsilon)$. Moreover, since all the $q_t$ are
Morse functions and since the Hessian of $q_t$ at $c_i(t)$ is a
continuous function of $t$, then for any $t \in [0,1]$, $c_i(t)$ is a
non-degenerate critical point of $q_t$ with same index as $c_i$.
\end{proof}

Consider now two smooth functions such that the critical points are close, as 
illustrated in~\figref{fig:criticalpts}.
Next we show that, in this circumstance,
the bottleneck distance takes a simple form.  

\begin{figure}
 \includegraphics[height=1.5in]{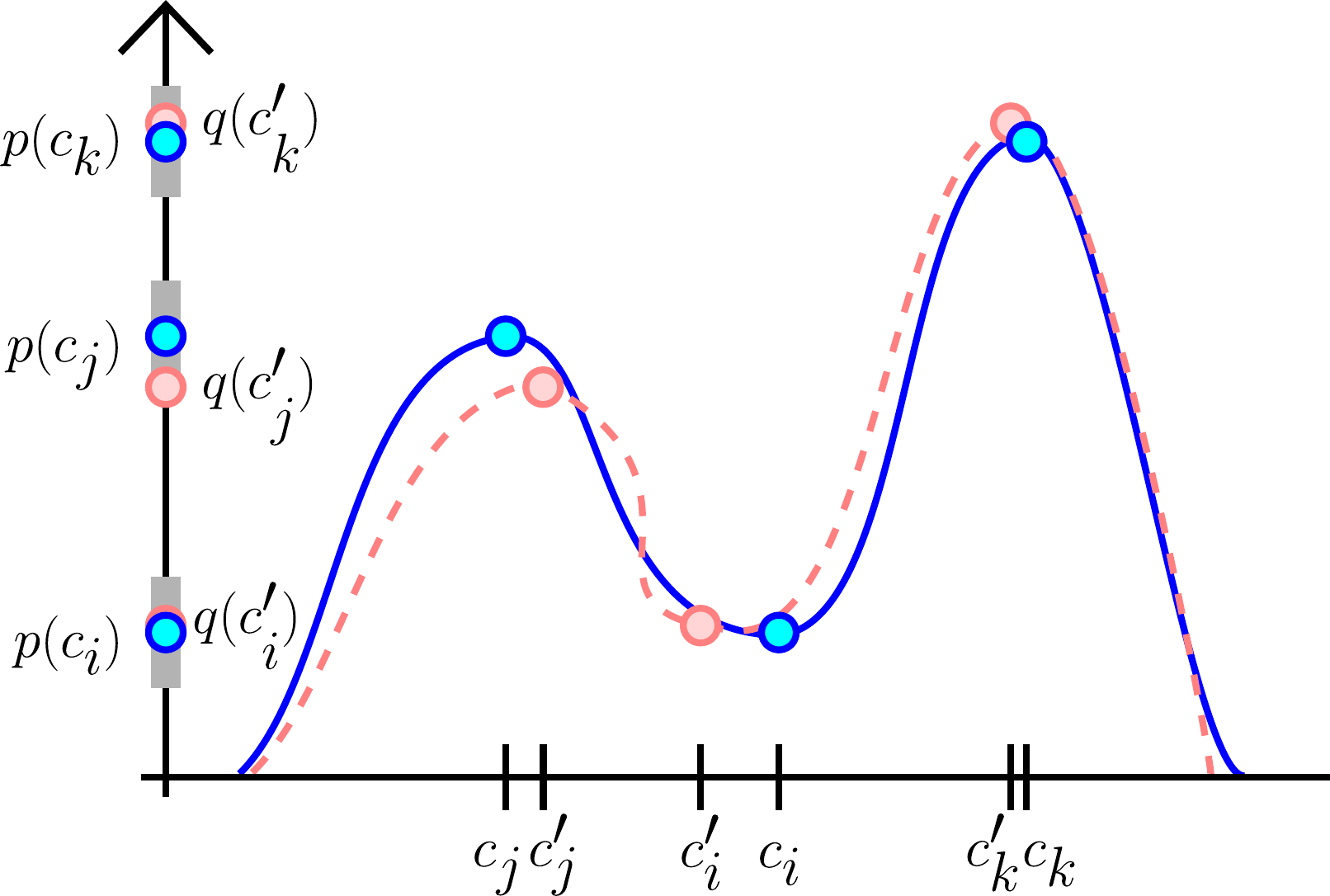}
 \caption{This figure illustrates the 
assumptions of Lemma \ref{lemma::bottle}.  The functions 
$p$ and $q$ are shown in solid blue and dashed pink, 
respectively.  The grey regions on the $y$-axis 
represent the sets $p(c) \pm b$ for critical points $c$ 
of $p$.}\label{fig:criticalpts}
\end{figure}

\begin{lemma}[Critical Distances]
\label{lemma::bottle}
Let $p$ and $q$ be two Morse functions as in Lemma \ref{lemma::stability},  with finitely many critical points
$C = \{c_1,\ldots, c_k\}$ and $C' = \{c_1',\ldots, c_k'\}$ respectively.
Let $D_p$ and $D_q$ be the persistence diagrams
from the upper level set filtrations of $p$ and $q$ respectively and let $a = \min_{i\neq j} |p(c_i)-p(c_j)|$
and $b = \max_j |p(c_j)-q(c_j')|$.
If $b \leq a / 2 - ||p - q||_{\infty}$ and $a/2 > 2 ||p - q||_{\infty}$, then $W_\infty(D_p,D_q) = b$.
\end{lemma}

\begin{proof}
The topology of the upper level sets of the Morse
functions $p$ and $q$ only changes at critical values (Theorem 3.1 in
\cite{Milnor1963}). As a consequence the non diagonal points of $D_p$
(resp. $D_q$) have their coordinates among the set $\{ p(c_1), \ldots,
p(c_k) \}$ (resp. $\{ q(c_1'), \ldots, p(c_k') \}$) and each $p(c_i)$
is the coordinate of exactly one point in $D_p$. Moreover, the
pairwise distances between the points of $D_p$ are lower bounded by
$a$ and all non diagonal points of $D_p$ are at distance at least $a$
from the diagonal.  From the persistence stability theorem
\cite{steiner05stable,chazal2012structure}, $W_\infty(D_p,D_q) \leq
||p - q||_{\infty}$. Since $a > 4 ||p - q||_{\infty}$ and $a \geq 2b +
2 ||p - q||_{\infty}$, the (unique) optimal matching realizing the
bottleneck distance $W_\infty(D_p,D_q)$ is such that if
$(p(c_i),p(c_j)) \in D_p$ then it is matched to the point
$(q(c_i'),q(c_j'))$ which thus have to be in $D_q$.  It follows that
$W_\infty(D_p,D_q) = b$.
\end{proof}

Now we establish the limiting distribution of
$\sqrt{n}W_\infty(\hat D,D)$.

\begin{theorem}[Limiting Distribution]
\label{thm::limit}
Let $p_h(x) = \mathbb{E}[\hat p_h(x)]$, where $\hat p_h(x)$ is the Kernel Density Estimator evaluated in $x$.
Assume that $p_h$ is a Morse function with
two uniformly bounded continuous derivatives
and finitely many critical points $c=\{c_1,\ldots, c_k \}$.
Let $D$ be the persistence diagram of the upper level sets of $p_h$
and let
$\hat D$ be the diagram of upper level sets of $\hat p_h$.
Then
$$
\sqrt{n} W_\infty(\hat D,D) 
\rightsquigarrow || Z||_\infty
$$
where 
$Z=(Z_1,\ldots, Z_k)\sim N(0,\Sigma)$
and
$$
\Sigma_{jk} =
\int K_h(c_j,u) K_h(c_k,u) dP(u) - 
\int K_h(c_j,u) dP(u)\, \int K_h(c_k,u) dP(u).
$$
\end{theorem}

\begin{proof}
Let 
$\hat c =\{\hat c_1,  \hat c_2, \dots \}$
be the set of critical points of
$\hat p_h$.
Let $g$ and $H$ be the gradient and Hessian of $p_h$.
Let $\hat g$ and $\hat H$ be the gradient and Hessian of $\hat p_h$.
By a standard concentration of measure argument
(and recalling that the support is compact), for any $\eta >0 $ 
there is an event $A_{n,\eta}$ such that,
on $A_{n,\eta}$,
\begin{equation}\label{eq::deriv-bound}
\sup_x ||\hat p^{(i)}_h(x) - p_h^{(i)}(x)|| < \eta
\end{equation}
for $i=0,1,2$,
and
$\mathbb{P}(A_{n,\eta}^c)\leq e^{-n c \eta^2}$. 
This is proved for $i=0$ in \cite{rao1983nonparametric}, \cite{gine2002rates}, \cite{yukich1985laws}, and the same proof gives the results for $i=1,2$.
It follows that
$\sup_x ||g(x) - \hat g(x)|| = O_P (1/\sqrt{n} )$ and
$\sup_x || H(x) - \hat H(x)||_{\rm max} = O_P(1/\sqrt{n} )$.

For $\eta $ smaller than a fixed value $\eta_0 $, we can apply Lemma \ref{lemma::stability}, we get that on
$A_{n,\eta}$, $\hat c$ and $c$ have the same number of elements and
can be indexed so that
\begin{eqnarray*}
\max_{j=1,\dots,k} \Vert \hat c_j - c_j \Vert &\leq&  \frac \eta C 
\end{eqnarray*}
where $C$ is the same constant is in Lemma \ref{lemma::stability}. We then take $\eta_n :=\sqrt{ \frac {\log n} {
n} }$ and we consider the events $A_n := A_{n,\eta_n}$.  Then, for $n$ large enough, on $A_n$ we get
$$
\max_{j=1,\dots,k} \Vert \hat c_j - c_j \Vert  = O\left(\sqrt{ \frac{\log n}{n}}\right)
$$
whereas $P \left(A_n^c \right) = o(1)$. In the following, we thus can restrict to 
$A_{n}$.

The critical values of $p_h$ are
$v = ( v_1 \equiv p_h(c_1),\ldots, v_k \equiv p_h(c_k))$ and
the critical values of $\hat p_h$ are
$\hat v = (\hat v_1\equiv \hat p_h(\hat c_1),\ldots, \hat v_k \equiv \hat p_h(\hat c_k))$.
Now we use Lemma \ref{lemma::bottle} to conclude that
$W_\infty(\hat D,D) = \max_j \Vert \hat v_j - v_j \Vert_\infty$ for $n$ large enough.
Hence,
$$
W_\infty(\hat D,D) = \max_{j=1,\dots,k} | \hat p_h(\hat c_j) - p_h(c_j)|.
$$
Then, using a Taylor expansion, for each $j$,
$$
\hat p_h(\hat c_j) = \hat p_h(c_j) + (\hat c_j - c_j)^T \hat g(c_j) + O(||\hat c_j - c_j||^2).
$$
Since $g(c_j) = (0,\ldots, 0)$ we can write the last equation as
$$
\hat p_h(\hat c_j) = \hat p_h(c_j) + (\hat c_j - c_j)^T (\hat g(c_j)-g(c_j)) + O(||\hat c_j - c_j||^2).
$$
So,
\begin{align*}
\sqrt{n}(\hat v_j - v_j) &= \sqrt{n}(\hat p_h(\hat c_j) - p_h(c_j))\\
&=
\sqrt{n}(\hat p_h(c_j) - p_h(c_j)) + 
\sqrt{n}(\hat c_j - c_j)^T (\hat g(c_j)-g(c_j)) +  O( ||\hat c_j - c_j||^2)\\
&= \sqrt{n}(\hat p_h(c_j) - p_h(c_j)) + 
\sqrt{n}(\hat c_j - c_j)^T (\hat g(c_j)-g(c_j)) + o(1/\sqrt{n}).
\end{align*}
For the second term, note that
$\sqrt{n}(\hat c_j - c_j)= O (\log n)$ and
$(\hat g(c_j)-g(c_j))= O_P(1/\sqrt{n})$.
So
$$
\sqrt{n}(\hat v_j - v_j) = \sqrt{n}(\hat p_h( c_j) - p_h(c_j)) + o_P(1).
$$
Therefore,
$$
\sqrt{n}W_\infty(\hat D,D) = 
\sqrt{n} \max_j | \hat v_j  - v_j |  = \max_j | \sqrt{n}(\hat p_h(c_j) - p_h(c_j)) | + o_P(1).
$$
By the multivariate Berry-Esseen theorem 
\citep{bentkus2003dependence},
$$
\sup_A | \mathbb{P}( \sqrt{n}(\hat p_h ( c) - p_h (c)) \in A) - \mathbb{P}( Z \in A)| \leq \frac{C_1}{\sqrt{n}}
$$
where the supremum is over all convex sets $A\in \mathbb{R}^k$,
$C_1>0$ depends on $k$ and the third moment of $h^{-d}K(x-X/h)$
(which is finite since $h$ is fixed and the support is compact),
$Z=(Z_1,\ldots, Z_k)\sim N(0,\Sigma)$
and
$$
\Sigma_{jk} =
\int K_h(c_j,u) K_h(c_k,u) dP(u) - 
\int K_h(c_j,u) dP(u)\, \int K_h(c_k,u) dP(u).
$$
Hence,
$$
\sup_t \Biggl|\mathbb{P}\Bigl( \max_j |\sqrt{n}(\hat p_h(c_j) - p_h(c_j)) | \leq t\Bigr) - 
\mathbb{P}( ||Z||_\infty \leq t)\Biggr| \leq \frac{C_1}{\sqrt{n}}.
$$
By Lemma
\ref{lemma::bottle},
$W_\infty(\hat D,D) = \max_j | \hat v_j - v_j |$.
The result follows.
\end{proof}

Let
$$
F_n(t) = \mathbb{P}(\sqrt{n}W_\infty(\hat D,D) \leq t).
$$
Let $X_1^*,\ldots,X_n^* \sim P_n$
where $P_n$ is the empirical distribution.
Let $\hat D^*$ be the diagram from
$\hat p_h^*$ and let
$$
\hat F_n(t) = \mathbb{P}\Bigl(\sqrt{n}W_\infty(\hat D^*,\hat D)\,\leq t \Bigm|\,X_1,\ldots,X_n\Bigr)
$$
be the bootstrap approximation to
$F_n$.

Next we show that the bootstrap quantity $F_n(t)$ converges to the same limit as
$F_n(t)$.

\begin{corollary}
Assume the same conditions as the last theorem. 
Then,
$$
\sup_t |\hat F_n(t) - F_n(t)|\stackrel{P}{\to} 0.
$$
\end{corollary}

\begin{proof}
The proof is essentially the same
as the proof of Theorem \ref{thm::limit}
except that $\hat p_h$ replaces $p_h$
and $\hat p_h^*$ replaces $\hat p_h$. Using the same notations as in the proof of Theorem~\ref{thm::limit}, we note
that on the set  $A_n$, for $n$ larger than a fixed value $n_0$, the function $\hat p_h$ is a Morse function with
two uniformly bounded continuous derivatives and finitely many critical points $\hat c=\{\hat
c_1,\ldots, \hat c_k \}$. We can restrict the analysis to sequence of events $A_n$ since $P(A_n)$ tends to zero.
Assuming that $ A_n$ is satisfied, using the same argument as in Theorem \ref{thm::limit}, we get that:
 $$
\sup_t \Biggl|
\mathbb{P}\Bigl( \max_j | \sqrt{n}(\hat p_h^*(\hat c_j) - \hat p_h(\hat c_j)) | \leq t
\,\Bigm|\, X_1,\ldots, X_n\Bigr) - \mathbb{P}( ||\tilde Z||_\infty \leq t)\Biggr| \leq \frac{C_2^*}{\sqrt{n}}
$$
where 
$\tilde Z\sim N(0,\hat \Sigma)$ with
$$
\hat \Sigma_{jk} =
\frac{1}{n}\sum_i  K_h(\hat c_j,X_i) K_h(\hat c_k,X_i) -
\frac{1}{n}\sum_i  K_h(\hat c_j,X_i)\, \frac{1}{n}\sum K_h(\hat c_k,X_i)
$$
and $C_2^*$ depends on the empirical third moments of   
$h^{-d}K((x-X^*)/h)$. There exists an upper bound $C_2$ on $C_2^*$, which only depends on $K$ and $P$. 
Since
$\max_{j,k}|\hat\Sigma_{j,k} - \Sigma_{j,k}| = O_P(\log n /\sqrt{n})$ and
$\max_j \Vert \hat c_j - c_j \Vert = O_P(\log n/\sqrt{n})$, we conclude that
$$
\sup_t| \mathbb{P}( ||\tilde Z||_\infty \leq t) - \mathbb{P}( ||Z||_\infty \leq t) | = 
O_P\left(\frac{\log n}{\sqrt{n}}\right).
$$
Then
\begin{align*}
\sup_t|\hat F_n(t) - F_n(t)| &\leq
\sup_t|\hat F_n(t) - \mathbb{P}( ||\tilde Z||_\infty \leq t)| \\
& +
\sup_t| \mathbb{P}( ||\tilde Z||_\infty \leq t) - \mathbb{P}( ||Z||_\infty \leq t) | +
\sup_t|F_n(t) - \mathbb{P}( || Z||_\infty \leq t)| = 
O_P\left(\frac{\log n}{\sqrt{n}}\right).
\end{align*}
The result follows.
\end{proof}

\section{Extensions}

In this section, we discuss how to deal with three issues that can arise:
choosing the parameters, correcting for boundary bias, and dealing with noisy
data.

\subsection{A Method for Choosing the Smoothing Parameter}
\label{section::choosing}

An unsolved problem in topological inference is how to
choose the smoothing parameter $m$ (or $h$).
\cite{guibas2013witnessed} suggested tracking
the evolution of the persistence of the homological features as the tuning parameter varies.
Here we make this method more formal,
by selecting the parameter that maximizes the total amount of significant persistence.

Let $\ell_1(m),\ell_2(m),\ldots, $
be the lifetimes
of the features at scale $m$.
Let $c_\alpha(m)/\sqrt{n}$
be the significance cutoff
at scale $m$.
We define two quantities that measure
the amount of significant information using parameter $m$:
$$
N(m) = \# \left\{i:\ \ell(i) > \frac{c_\alpha(m)}{\sqrt{n}}\right\},\ \ \ 
S(m) = \sum_i \left[ \ell_i - \frac{c_\alpha(m)}{\sqrt{n}}\right]_+.
$$
These measures are small when $m$ is small since $c_\alpha(m)$ is large.
On the other hand,
they are small when $m$ is large since then all the features are smoothed out.
Thus we have a kind of topological bias-variance trade-off.
We choose $m$ to maximize $N(m)$ or $S(m)$.
The same idea can be applied to the kernel distance and kernel density estimator.
See the example in Figure~\ref{fig:methodCassini}.

\begin{figure}[!ht]
\begin{center}
\includegraphics[scale=0.155]{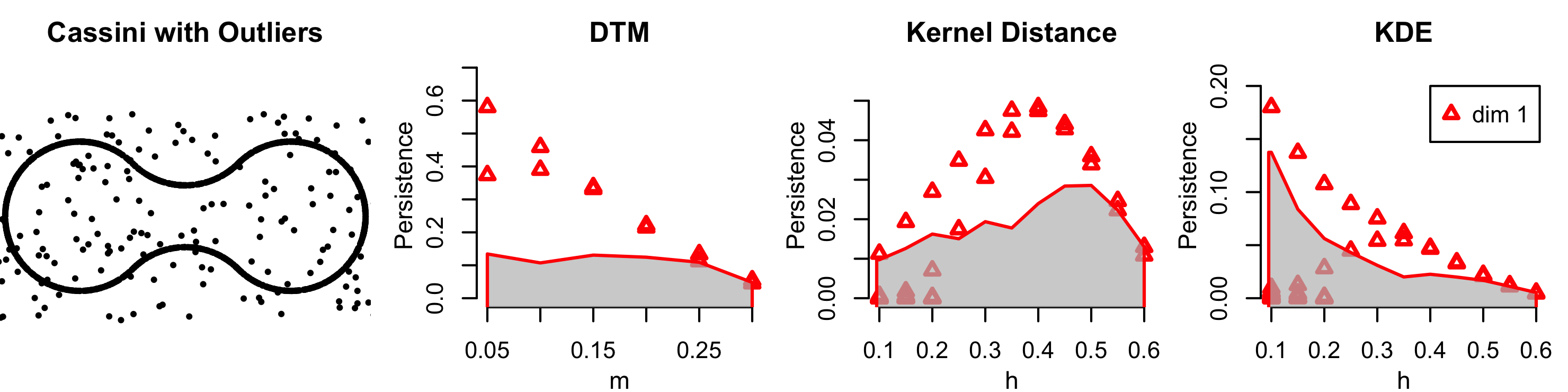}
\end{center}
\caption{Max Persistence Method with Bottleneck Bootstrap Bands for  1-dimensional features. 
DTM: $\argmax_m N(m)= \{0.05, 0.10, 0.15, 0.20\} $, 
$\argmax_m S(m)= 0.05$ ; Kernel Distance: $\argmax_h N(h)= \{0.25, 0.30, 0.35, 0.40, 0.45, 0.50\}$, 
$\argmax_h S(h)= 0.35$; KDE: $\argmax_h N(h)= \{ 0.25, 0.30, 0.35, 0.40, 0.45, 0.50\}$, $\argmax_h S(h)= 0.3$ 
The plots show how to choose the smoothing parameters to maximize the number of 
significant features.}
\label{fig:methodCassini}
\end{figure}

\subsection{Boundary Bias}\label{ss::bias}

It is well known that kernel density estimators
suffer from boundary bias.
For topological inference,
this bias manifests itself in a particular form
and the same problem affects the DTM.
Consider Figure~\ref{fig::boundary}.
Because of the bounding box,
many of the loops are incomplete.
The result is that, using either the DTM or the KDE
we will miss many of the loops.

There is a large literature on reducing boundary bias in the kernel
density estimation literature.
Perhaps the simplest approach is to reflect the data around the boundaries (see for example \cite{schuster1958incorporating}).
But there is a simpler fix for topological inference: we merely need to close the
loops at the boundary.
This can be done by adding points uniformly around the boundary.

\begin{figure}[!ht]
\begin{center}
\includegraphics[scale=0.89]{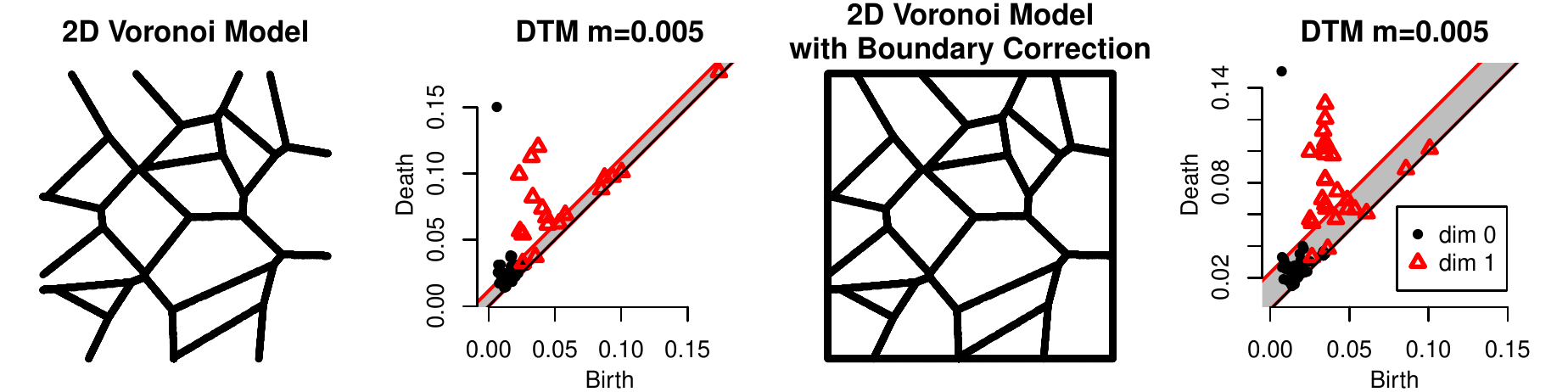}
\end{center}
\caption{First: 10,000 points sampled from a 2D Voronoi model with 20 nuclei.  
Second: the corresponding persistence diagram of sublevel sets of the distance 
to measure function. Note that only 9 loops are detected as significant. Third: 
2,000 points have been added on the boundary of the square delimiting the 
Voronoi model. Fourth: now the corresponding persistence diagram shows 16 
significant loops.}
\label{fig::boundary}
\end{figure}


\subsection{Two Methods for Improving Performance}

We can improve the performance of all the methods
if we cam mitigate the outliers and noise.
Here we suggest two methods to do this.
We focus on the kernel density estimator.

First, a simple method to reduce the number of outliers
is to truncate the density, that is,
we eliminate 
$\{X_i:\ \hat p(X_i) < t\}$
for some threshold $t$.
Then we re-estimate the density.

Secondly, we sharpen the data as
described in
Choi and Hall (1999) and
Hall and Minnotte (2002).
The idea of sharpening
is to move each data point $X_i$ slightly in the direction of the gradient
$\nabla \hat p(X_i)$ and then re-estimate the density.
The authors show that this reduces the bias at peaks in the density
which should make it easier to find topological features.
It can be seen that the sharpening method
amounts to running one or more steps if the
mean-shift algorithm.
This is a gradient ascent which is intended to find modes
of the density estimator.
Given a point $x$, we move $x$ to 

$$
\frac{\sum_i X_i K_h(x,X_i)}{\sum_i K_h(x,X_i)},
$$
which is simply the local average centered at $x$.
For data sharpening, we do one (or a few)
iterations of this to each data point $X_i$.
Then the density is re-estimated.\\
In fact, we could also use the
subspace constrained mean shift algorithm
(SCMS) which moves points towards ridges of the density;
see \cite{ozertem2011locally}.\\
Figure~\ref{fig::sharpen} shows these methods applied
to a simple example.

\begin{figure}[!ht]
\begin{center}
\includegraphics[scale=0.89]{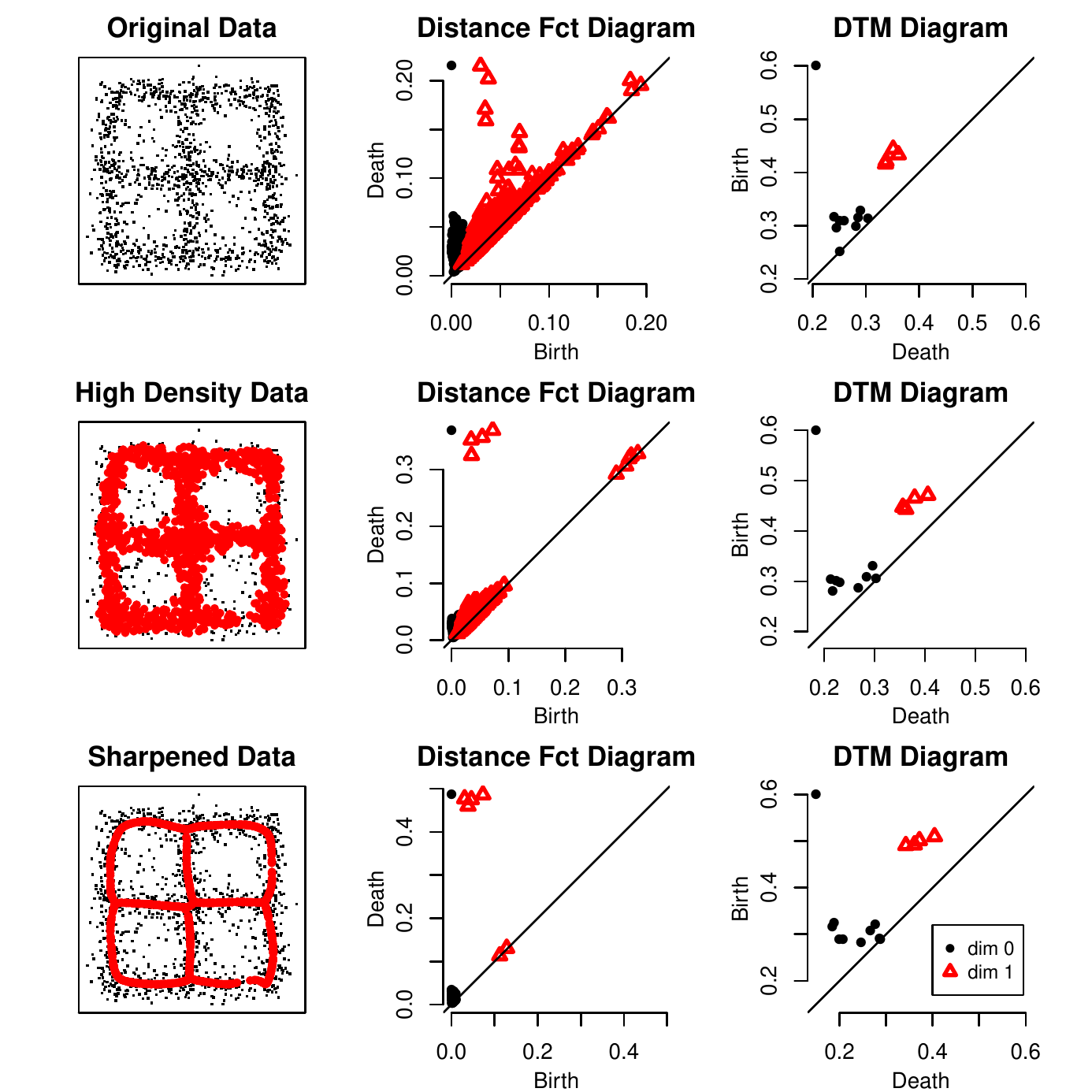}
\end{center}
\caption{Top: 1,300 points sampled along a $2 \times 2$ grid  with Gaussian
noise; the diagram of the distance function shows many loops due to noise.
Middle: the red points are the high density data (density $> 0.15$);  the
corresponding diagram of the distance function correctly captures the 4 loops,
plus a few features with short lifetime. Bottom: the red points represent the
sharpened high density data; now most of the noise in the corresponding diagram
is eliminated. Note that the diagram of the distance to measure function does a
good job with the original data. The bottom left plot shows a slight
improvement, in the sense that the persistence of the 4 loops has increased.}
\label{fig::sharpen}
\end{figure}

\newpage

\section{Examples}
\label{sec:examples}

\begin{example}[Noisy Grid]
\label{ex:noisy-grid}
The data in Figure~\ref{fig::noisy-grid}
are 10,000 data points on a 2D grid. We add Gaussian noise plus 1,000 outliers and compute the persistence diagrams of Kernel Density Estimator, Kernel distance, and Distance to Measure. The pink bands show 95\% confidence sets obtained by bootstrapping the corresponding functions. The black lines show 95\% confidence bands obtained with the bottleneck bootstrap for dimension 0, while the red lines show 95\% confidence bands obtained with the bottleneck bootstrap for dimension 1. The Distance to Measure, which is less sensitive to the density of the points, correctly captures the topology of the data. The Kernel Distance and KDE find some extra significant connected component, corresponding to high density regions at the intersection of the grid.
\end{example}

\begin{figure}[!ht]
\begin{center}
\includegraphics[scale=0.89]{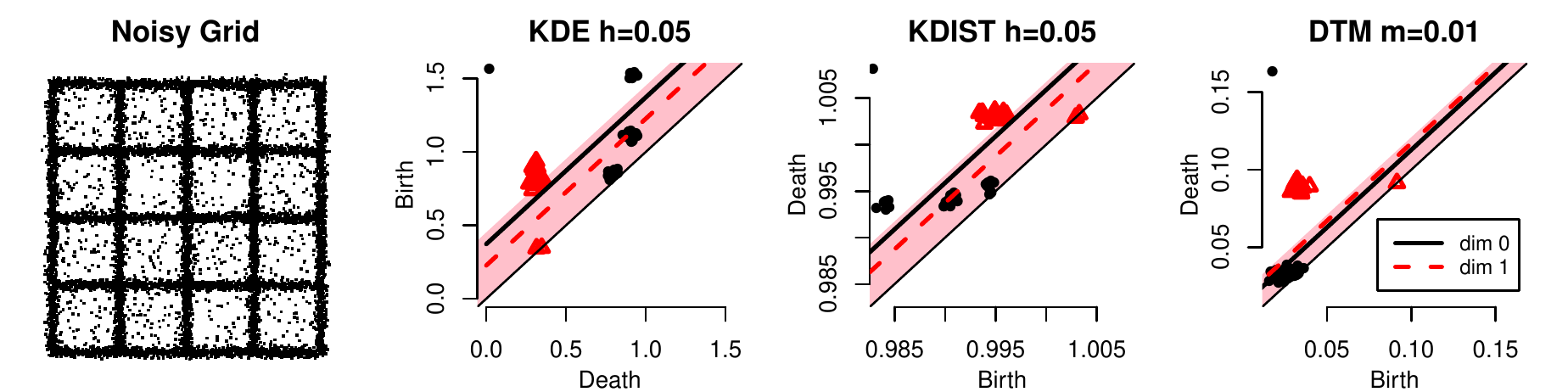}
\end{center}
\caption{10,000 data points on a 2D grid and the corresponding persistence diagrams of Kernel Density Estimator, Kernel distance, and Distance to Measure. For more details see Example \ref{ex:noisy-grid}.
}
\label{fig::noisy-grid}
\end{figure}

\begin{example}[Soccer]
Figure~\ref{fig::soccer} shows the field position of
two soccer players. The data come from body-sensor traces collected
during a professional soccer game in late 2013 at the Alfheim Stadium
in Tromso, Norway. The data are sampled at 20 Hz. See
\cite{pettersen2014soccer}.  
Although the data is a function observed over time,
we treat it as a point cloud.
Points
on the boundary of the field have been added to avoid boundary
bias.
The DTM captures the difference between
the two players: the defender leaves one big portion of the filed
uncovered (1 significant loop in the persistence diagram), while the
midfielder does not cover the 4 corners (4 significant loops). 
Nonetheless, the Kernel distance, which is more sensible to the
density of these points, fails to detect significant topological
features.
\end{example}

\begin{figure}[!ht]
\begin{center}
\includegraphics[scale=0.31]{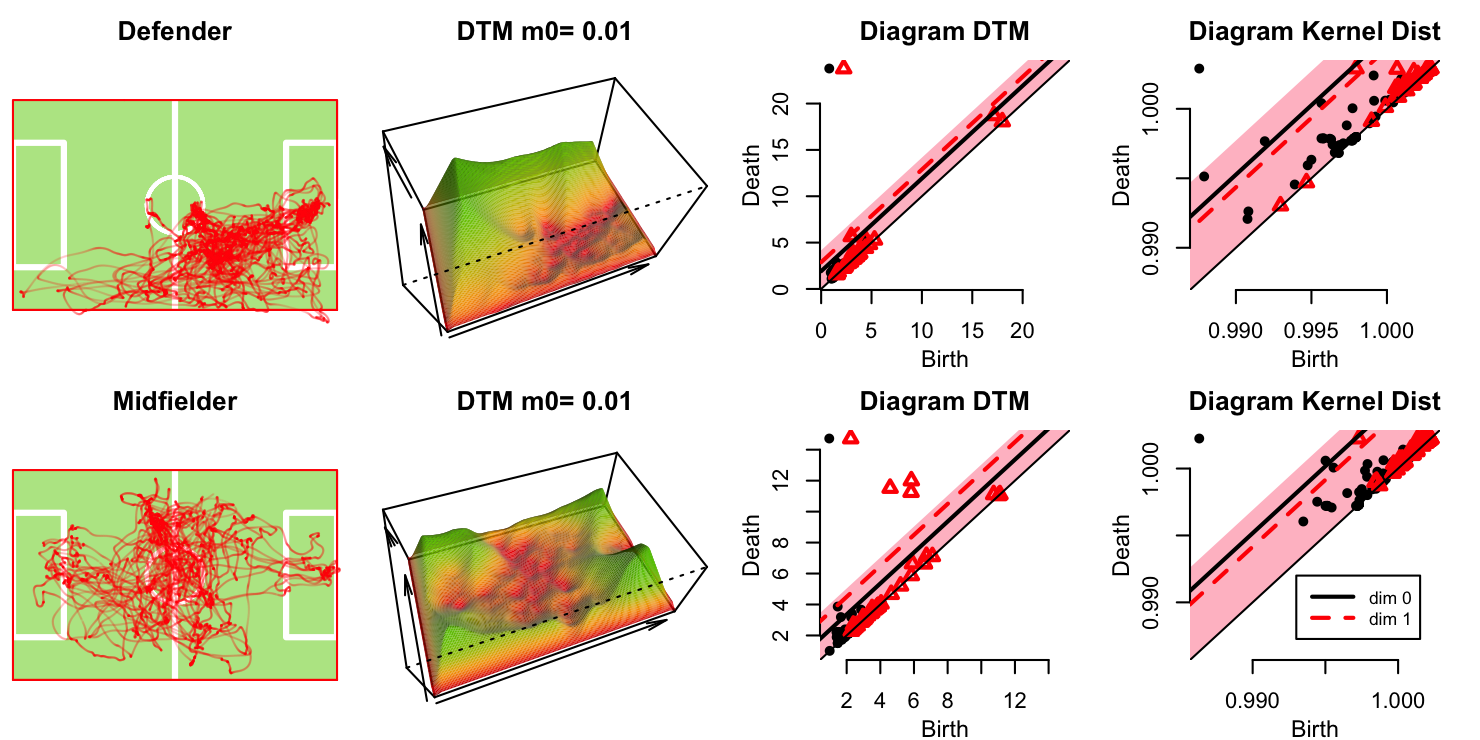}
\end{center}
\caption{Top: data for a defender. We show the DTM, the digram for the DTM and the digram
for the kernel distance.
Bottom: same but for a midfielder.
The midfielder data has more loops.}
\label{fig::soccer}
\end{figure}

\begin{example}[Voronoi Models]
\label{ex:voronoi}
Given $k$ points (nuclei) $\{z_1, \dots, z_k\} \subset \mathbb{R}^3$, let the Voronoi region $R_k$ be
$
R_k=\left\{x \in \mathbb{R}^3: \Vert x - z_k  \Vert \leq \Vert x - x_j \Vert \text{ for all  } j\neq k  \right\}.
$
The Voronoi regions $R_1, \dots, R_k$ partition the space, forming what is known as the Voronoi diagram.
A face is formed by the intersection of 2 adjacent Voronoi regions; a line is formed at the intersection of two faces 
and a node is formed at the intersection of two or more lines.

We will sample points around the the nodes, lines and faces that are formed at the intersection of the Voronoi regions.
A Voronoi wall model is a sampling scheme that returns points within or around the Voronoi faces.
Similarly, by sampling points exclusively around the lines or exclusively around the nodes, we can construct
 Voronoi filament models and  Voronoi cluster models.

These models were introduced by
\cite{icke1991galaxy} to mimic key features
of cosmological data;
see also \cite{van2011alpha}.

In this example we generate data from filament models and wall models
using the basic definition of Voronoi diagram, computed on a fine grid in $[0,50]^3$.
We also add random Gaussian noise.
See Figure~\ref{fig:Rien}: the first two rows
show 100K particles concentrated around the filaments of 8 and 64
Voronoi cells, respectively. The last two rows show 100K particles
concentrated around the walls of 8 and 64 Voronoi cells. 60K points on the
boundary of the boxes have been added to mitigate boundary bias.
For each model we present the persistent diagrams of the distance function, distance to measure and kernel density estimator.
We chose the smoothing parameters by maximizing the quantity $S(\cdot)$, defined in Section \ref{section::choosing}.

The diagrams illustrate the evolution of the filtrations for the three different functions: at first, the connected components appear (black points in the diagrams); then they merge forming loops (red triangle), that eventually evolve into 3D voids (blue squares).

The persistence diagrams of the three functions allow us to distinguish the 
different models (see Figure~\ref{fig:introExample} for a less trivial example) 
and the confidence bands, generated using the bootstrap method of Section 
\ref{ss:significance}, allow us to separate the topological signal from the 
topological noise.
In general, the DTM performs better than the KDE, which is more affected by the high density of points around the nodes and filaments.
For instance, this is very clear in the third row of Figure~\ref{fig:Rien}. The 
DTM diagram correctly captures the topology of the Voronoi wall model with 8 
nuclei: one connected component and 8 voids are significant, while the remaining 
homological features fall into the band and are classified as noise.
\end{example}

\begin{figure}[!ht]
\begin{center}
\includegraphics[scale=0.315]{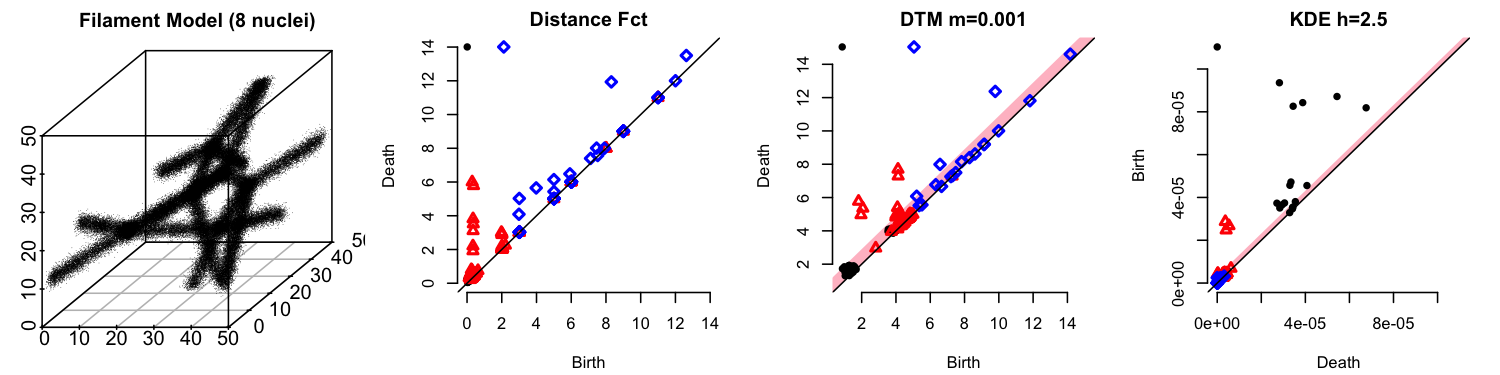}
\includegraphics[scale=0.315]{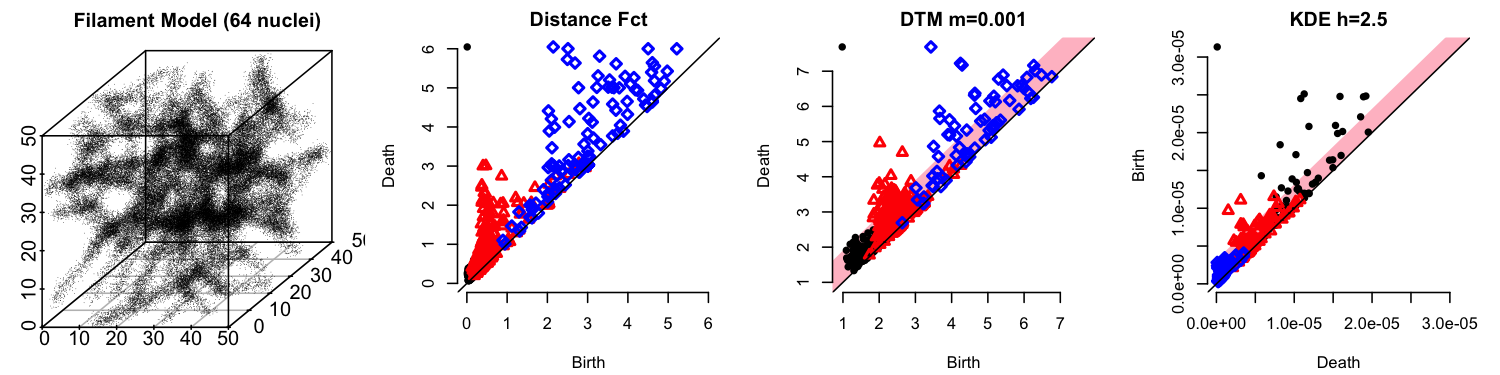}
\includegraphics[scale=0.315]{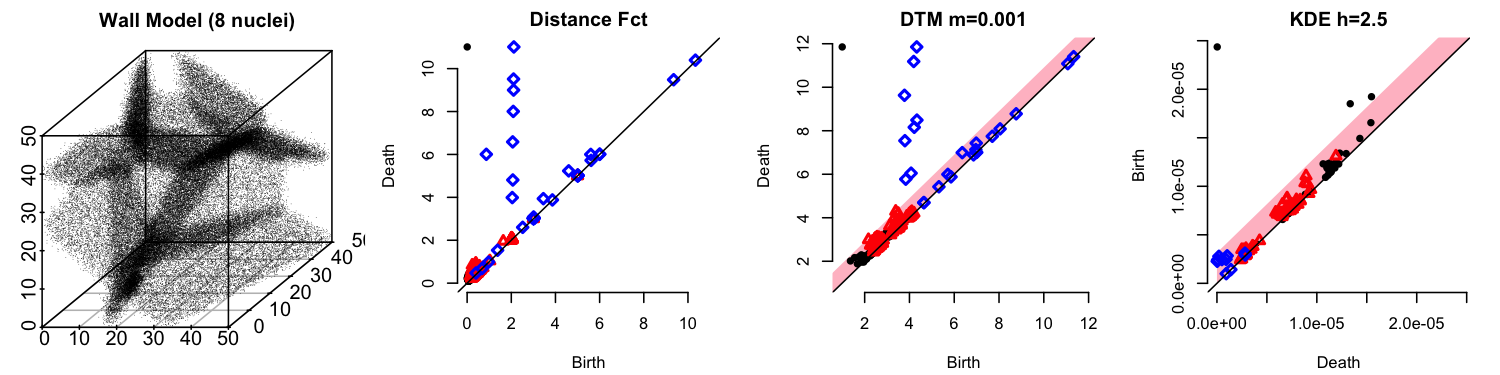}
\includegraphics[scale=0.315]{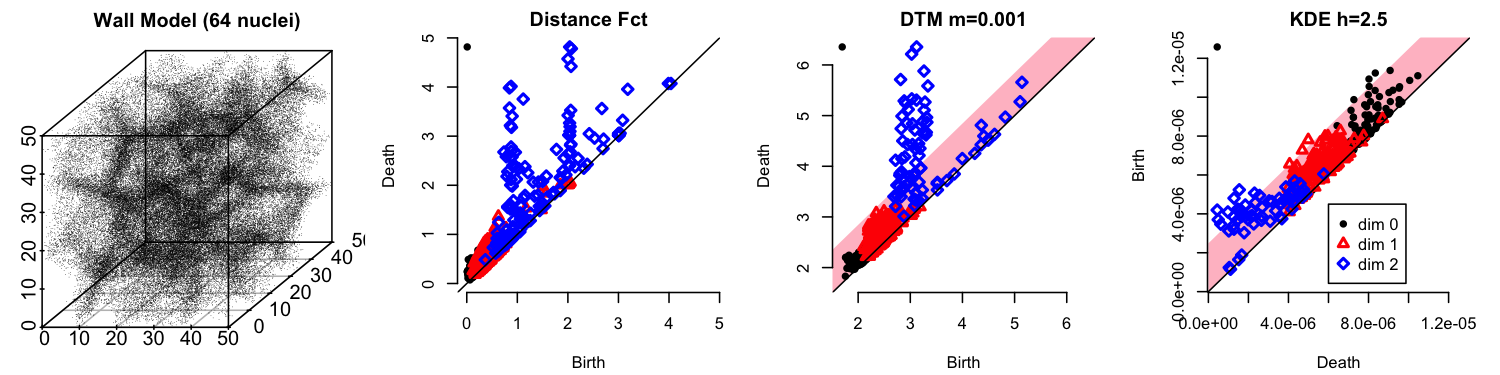}
\end{center}
\caption{
Data from four Voronoi foam models.
In each case we show the diagrams of the
distance function, the DTM and the KDE.
A boundary correction was included.}
\label{fig:Rien}
\end{figure}

\section{Discussion}

In this paper, we showed how the DTM and KDE
can be used for robust topological inference.
Further, we showed how to use the bootstrap to 
identify topological features that are distinguishable 
from noise.
We conclude by discussing two issues:
comparing DTM and KDE, and
using persistent homology versus selecting a single level set.

\subsection{Comparison of DTM and Kernel Distance}
\label{sec::compare}

The DTM and the KDE have the same broad aim:
to provide a means for extracting topological features from data.
However, these two methods are really focused on different goals.
Consider again the model
$P = \pi R + (1-\pi) (Q\star \Phi_\sigma)$
and let $S$ be the support of $Q$.
As before,
we assume that $S$ is a ``small set'' meaning that either it has dimension $k < d$
or that it is full dimensional but has small Lebesgue measure.
When $\pi$ and $\sigma$ are small,
the persistent homology of the upper level sets of the density $p$
will be dominated by features corresponding to the homology of $S$.
In other words, we are using the persistent homology of
$\{ p > t\}$ to learn about the homology of $S$.
In contrast, the DTM is aimed at estimating the persistent homology of $S$.
Both are useful, but they have slightly different goals.

This also raises the intriguing idea of extracting more information
from both the KDE and DTM by varying more parameters.
For example, if we look at the sets
$\{p_h > t\}$ for fixed $t$ but varying $h$,
we get information very similar to that of the DTM.
Conversely, for the DTM, we can vary the tuning parameter $m$.
There are many possibilities here which we will investigate in future work.

\subsection{Persistent Homology Versus Choosing One Level Set}

We have used the persistent homology of the upper level sets
$\{\hat p_h > t\}$ to
probe the homology of $S$.
This is the approach used in
\cite{bubenik2012statistical}
and \cite{phillips2014goemetric}.

Bobrowski et al (2104)
suggest a different approach.
They select a particular level set
$\{ p > t\}$
and they form a robust estimate of the homology
of this one level set.
They have a data-driven method for selecting $t$.
(This approach is only one part of the paper.
They also consider persistent homology.)

They make two key assumptions.
The first is that there exists $A<B$ such that
$\{ p> t\}$ is homotopic to $S$ for all
$A < t < B$.
(If two sets are homotopic, then they have the same homology.)
This is a very reasonable assumption.
In the mixture model
$P = \pi R + (1-\pi) (Q\star \Phi_\sigma)$
this assumption will be satisfied when $S$ is a small set
and when $\pi$ and $\sigma$ are small.
In this case, persistent homology will also work well:
the dominant features in the persistence diagram
will correspond to the homology of $S$.

Bobrowski et al (2104) make an additional assumption.
They assume that the dimension $k$ of $S$ is known and that
the rank of the $k^{\rm th}$ homology group is 0 for all
$t > B$.
This assumption is critical for their approach to choosing a single level set.
Currently,
it is not clear how strong this
assumption is.
In future work, we plan to compare the robustness of the single-level approach
versus persistent homology.

\subsection{Future Work}

Lastly, we would like to mention that
several issues deserve future attention.
In particular, the methods we discussed for
choosing the tuning parameters, for mitigating boundary bias
and for sharpening the data, all deserve further investigation.

In a companion paper we will show how the ideas presented in this work can be used to
develop hypothesis tests for comparing point clouds.

\section*{Acknowledgements}
The authors are grateful to J\'erome Dedecker for pointing out  the key decomposition \eqref{decompos:AnRn} of the DTM.
The authors also would like to thank Jessi Cisewski and Jisu Kim for their comments.


\bibliographystyle{ims}
\bibliography{paper}

\end{document}